\documentclass[12pt]{article}

\usepackage[T2A, T1]{fontenc}
\usepackage[utf8]{inputenc}
\usepackage[russian,english]{babel}
\usepackage{substitutefont}
\usepackage{lmodern}
\substitutefont{T2A}{lmr}{fcm}
\substitutefont{T2A}{lmss}{fcs}
\substitutefont{T2A}{lmtt}{fct}

\usepackage{amsthm,amsopn,amsmath,amssymb}
\usepackage{graphicx}

\usepackage{indentfirst}
\usepackage[top=1.5cm,bottom=1.5cm,left=1.5cm,right=1.5cm]{geometry}
\usepackage{comment}
\usepackage{microtype}

\usepackage[style=numeric,backend=biber,autolang=other,sorting=none]{biblatex}
\renewbibmacro{in:}{}
\usepackage[style=russian]{csquotes}
\usepackage{makecell}
\usepackage{enumitem}

\newcommand{\R}{\mathbb{R}}
\newcommand{\eqdef}{\stackrel{\mathrm{def}}{=}}
\newcommand{\ddt}{\dfrac{d}{dt}}
\newcommand{\const}{\mathrm{const}}

\DeclareMathOperator{\Int}{\mathrm{int}}
\DeclareMathOperator{\dom}{\mathrm{dom}}
\DeclareMathOperator{\epi}{\mathrm{epi}}

\newtheorem{lemma}{Lemma}
\newtheorem{thm}{Theorem}
\newtheorem{corollary}{Corollary}
\newtheorem{remark}{Remark}
\newtheorem{defn}{Definition}
\newtheorem{prop}{Proposition}
\newtheorem{hyp}{Hypothesis}

\begin{document}

\title{THE ANALYTICAL SOLUTION TO NEWTON's AERODYNAMIC PROBLEM IN THE CLASS OF BODIES WITH VERTICAL PLANE OF SYMMETRY AND DEVELOPABLE SIDE BOUNDARY}

\author{L.~V.~Lokutsievskiy and M.~I.~Zelikin}
\date{}

\maketitle

\begin{abstract}
	The method of Hessian measures is used to find the differential equation that defines the optimal shape of nonrotationally symmetric bodies with minimal resistance moving in a rare medium. The synthesis of optimal solutions is described. A theorem on the optimality of the obtained solutions is proved.
\end{abstract}

\section{Introduction}

The goal in the present paper is to obtain an exact analytic expression for the shape of bodies exhibiting minimal resistance while moving in rarefied air surroundings. For axially symmetric convex bodies, this problem was proposed and solved by Sir Isaak Newton~\cite{Newton}.

Let the shape of the body be given by a convex function $z=u(x_1,x_2)$. Then the resistance is calculated by the formula
\[
\mathcal{J}(u) = \int_{\Omega} \frac{1}{1+u_{x_1}^2 + u_{x_2}^2}\,dx_1\wedge dx_2\to\min_u,
\]
\noindent where $\Omega$ is the support of the function $u$. The optimal shape is searched among convex bodies, since the convexity condition guarantees that the collision of each particle with a body in a rare medium is unique (the case of multiple collisions was considered in~\cite{Plakhov}). So the solution is sought in the class of convex functions with given support $\Omega$ having height $M>0$ (i.e., $-M\le u|_{\Omega}\le 0$).

At the very end of the 20th century, the problem was considered for bodies that are not surfaces of revolution~\cite{ButazzoKawohl1993}. It was shown~\cite{BrockFeroneKawohl1996} that the removal of the hypothesis of axial symmetry allows reducing resistance: nonaxially symmetric bodies with less resistance than symmetric ones of the same length and cross-section were found.

The exact form of the best shape of bodies with minimal resistance is still unknown. It was considered as a challenge for experts in optimal control theory. In this paper, we try to deal with this challenge.

The main difficulty in solving this optimization problem is the following. It is known~\cite{BrockFeroneKawohl1996,LachandPeletier2001} that if $u$ is an optimal solution that is  $C^2$-smooth in a subregion $\omega\Subset\Omega$, then $\det u''\equiv0$ on $\omega$.  So for a local $C^2$-variation $v$ with support in $\omega$, the function $u+\lambda v$ will be nonconvex for all small $\lambda$. Hence the optimal solution $u$ does not need to satisfy the Euler--Lagrange equation for the functional $\mathcal{J}$. Indeed, the condition $\det u''=0$ means that the surface $z=u(x_1,x_2)$ on each of its smooth parts must be developable. 

Thus, we have no tools to describe the optimal solutions of the problem. At the same time, it is known~\cite{ButazzoKawohl1993}, that optimal solutions exist. The necessary condition that $\det u''=0$ in any region of smoothness of~$u$ means that the surface $z=u(x_1,x_2)$ must be developable for $(x_1,x_2)$ in this region.

In this paper, we replace the hypothesis of axial symmetry by the less restrictive hypotheses of (i) mirror symmetry wrt a vertical plane and (ii) developable structure of the side boundary. Let us remark that all existing aircraft and ships, to say nothing of living creatures, have such symmetry. We propose the method of Hessian measures that allows us to describe explicitly the shape of the body in this case.

It appears that these shapes are defined as solutions of a special differential equation contiguous to equations of Painlev\'e type. The optimality of the obtained solutions is proved.

The paper has the following structure.
\begin{itemize}
	\item In Sec.~2, we give the statement of the problem in terms of convex analysis. In Sec.~3, we describe and develop the method of Hessian measures.
	\item In Sec.~4, using this method, we reduce the problem to an optimal control problem with one-dimensional control on the half-line. We shall call it the key problem, since it appears that this problem plays a key role in finding the optimal shape of the body. The specificity of the key problem is the presence of a singular point at the right end. This singularity has the structure of a movable pole with ramifications of second order (in terms of complex analysis).
	\item In Sec.~5, we prove the $C^1$-smoothness of solutions to the key problem.
	\item In Sec.~6, the equation for singular extremals is deduced from Pontryagin's Maximum Principle in the key problem. This equation has a movable critical singularity (in terms of complex differential equations) at the right end, which is contiguous to equations of Painlev\'e type.	
	\item In Sec.~7, the structure of solutions to this equation is investigated in detail.
	\item In Sec.~8, the field of extremals in the key problem is built.
	\item In Sec.~9, the local optimality of extremals of the field is proved.
	\item In the final Sec.~10, the obtained solutions of Newton's aerodynamic problem are described and a hypothesis is proposed. 
\end{itemize}

\section{Statement of the problem in terms of convex analysis}

Let us give an exact mathematical statement of Newton's aerodynamic problem in terms of convex analysis. Let $\R^n$ be Euclidean space, $\Omega\subset\R^n$ a convex compact set with nonempty interior, $M\ge 0$, and $\delta_\Omega(x)$ the indicator function of the set $\Omega$, i.e., $\delta_\Omega(x)=0$ for $x\in\Omega$ and $\delta_\Omega(x)=\infty$ for $x\not\in\Omega$. Let us denote by $C_M$ the set of all convex closed functions $u:\R^n\to \R$, such that $\delta_\Omega-M\le u\le \delta_\Omega$. In other words, $\dom u=\Omega$ and, for $x\in\Omega$, the condition $-M\le u(x)\le 0$ is valid. We must minimize on~$C_M$ the following functional
\begin{equation}
\label{problem:start}
\mathcal{J}(u) = \int_{\Omega}\frac{1}{1+|u'(x)|^2}\,dx \to \min,\quad u\in C_M.
\end{equation}
\noindent (for a convex function $u$, the derivative $u'$ exists almost everywhere).

The functional $\mathcal{J}$ defines the resistance of a convex $(n+1)$-dimensional solid body of the form $\epi u$ (or $\epi u\cap\{z\le 0\}$) in a constant vertical rarefied flow of particles (moving upwards). The first interesting case is the three-dimensional body corresponding to $n=2$.

This work is a continuation of the paper~\cite{LZ2018} in which we used a new approach to investigate Newton's aerodynamic problem. Our approach was based on the Hessian measures (see~\cite{ColesantiFirst}) and on the transition to the conjugate problem. In this work, we obtain explicit solutions of Newton's aerodynamic problem by using this machinery.

\section{Hessian measures}

The main idea of our approach is to make the change of variable $x\mapsto p(x)$ in the integral \eqref{problem:start}, where the change $x\mapsto p(x)$ is obtained from the Legendre--Young transformation $u^*(p)=\sup_{x}\big(\langle p,x\rangle - u(x)\big)$ of a convex function $u$. Generally speaking, the classical Legendre transformation defines the mapping $x\mapsto p$ only if $u\in C^1$. In the general case $u\not\in C^1$, the mapping $x\mapsto p$ is multivalued and the direct change is impossible. Nevertheless, due to works of Colesanti and Hug \cite{ColesantiFirst,ColesantiHug}, it is possible to make such a change in integral \eqref{problem:start}, since the Lebesgue measure $L^n$ on $\R^n$ becomes the Hessian measure $F_0$ on $\R^{n*}$ defined by the conjugate function $u^*$. 

Let us give a short clarification about the Hessian measures. In~\cite{ColesantiFirst}, it was proved that a convex function having an effective domain with nonempty interior (in our case $u^*$) in $n$-dimensional Euclidean space defines on it the following system of Borel measures $F_j$, $j=0,\ldots,n$. Let $\eta\subset \R^{n*}$ be a Borel set. For any $\varepsilon>0$, we define\footnote{Since all is defined on Euclidean space, the dot product gives the canonical isomorphism $\R^n\simeq\R^{n*}$, and we can assume that $\partial u^*(p)\subset \R^{n*}$.}
\[
\eta^\varepsilon = \bigcup_{p\in\eta} \big(p+\varepsilon\partial u^*(p)\big).
\]
\noindent Then the volume $\eta^\varepsilon$ is a polynomial in $\varepsilon$, i.e., the following analog of Steiner's formula (see~\cite{Schneider}) is valid:
\begin{equation}
\label{eq:schteiner_hessian_mesuare}
L^n(\eta^\varepsilon) = \sum_{j=0}^n \binom{n}{j} F_{n-j}(\eta|u^*) \varepsilon^{j},
\end{equation}
\noindent where $L^n$ is Lebesgue measure.

\begin{defn}
	The measures $F_j(\,\cdot\,|u^*)$ on $\R^{n*}$ defined by a convex function $u^*$ are called \textit{Hessian measures} of the function $u^*$.
\end{defn}

The general construction of Hessian measures is given in \cite{ColesantiHug}. We only note that if $u^*\in C^2$ on a domain $U$, then, for any Borel subset 
$\eta \subset U$, the following formula is fulfilled:
\[
\binom{n}{j} F_{j}(\eta|u^*) = \int_\eta S_{n-j}(p) dp,
\]
\noindent where $S_j(p)$ denotes the elementary symmetric polynomial of degree $j$,
\[
S_j(p) = \sum_{1\le k_1 < \ldots<k_j\le n} \lambda_{k_1}(p)\ldots\lambda_{k_j}(p),
\]
\noindent in the eigenvalues $\lambda_1(p),\ldots,\lambda_n(p)$ of the Hessian form $(u^*)''(p)$.

It is easy to see that $F_n\equiv L^n$. Moreover, it was proved in \cite{ColesantiFirst} that
\[
F_0(\eta|u^*) = L^n\Big( \bigcup_{p\in\eta}\partial u^*(p) \Big) = 
L^n\Big\{ x: \partial u(x)\cap\eta\ne\emptyset \Big\}
\]
\noindent (more general relations between Hessian measures of a function and its conjugate were given in~\cite[Theorem 5.8]{ColesantiHug}).

In~\cite{LZ2018}, the following result on the Legendre--Young--Fenchel transformation in problem~\eqref{problem:start} was obtained.
\begin{thm}[\cite{LZ2018}, Theorem 1]
	\label{thm:conjugate_problem}
	The transformation $u\mapsto u^*$ bijectively maps the class $C_M$ onto the class $C_M^*$ of convex functions $u^*$ on $\R^{n*}$ so that $s_\Omega\le u^*\le s_\Omega+M$ (where $s_\Omega$ denotes the support function of the set $\Omega$). Meanwhile,
	\[
	\mathcal{J}(u)=\mathcal{J}^*(u^*)\eqdef \int_{\R^{n*}} \frac{1}{1+|p|^2} F_0(dp|u^*).
	\]
\end{thm}

Thus, in the conjugate space $\R^{n*}$, Newton's aerodynamic problem is formulated as follows:
\[
\mathcal{J}^*(u^*)\to\min\qquad u^*\in C_M^*.
\]

\noindent This problem will be called conjugate to problem \eqref{problem:start}. 

The following theorem was obtained in~\cite{LZ2018}. This theorem is very useful in proving the existence of solutions in both the class $C_M$ and its subclasses:

\begin{thm}[\cite{LZ2018}, Theorem 2]
	\label{thm:limit}
	Suppose that a sequence $u_k\in C_M$ converges pointwise in the interior $\Int\Omega$ to a function $u\in C_M$. Also let $f:\R^{n*}\to\R$ be a bounded and continuous function. Then
	\[
	\lim_{k\to\infty}\int_{\Omega}f(u_k'(x))\,dx =
	\int_{\Omega}f(u'(x))\,dx.
	\]
	
	\noindent and all the integrals are well defined and finite.
\end{thm}

For example, if one takes a minimizing sequence in the class $C_M$ and chooses a pointwise converging subsequence from it (using a standard procedure), then Theorem~\ref{thm:limit} gives am immediate proof of the existence of an optimal solution in $C_M$. The existence theorem was proved for the first time in \cite{Marcellini1990,Buttazzo1995} (by using another method).

\section{The Maxwell stratum}

This paper is devoted to the construction of an explicit solution to Newton's aerodynamic problem in the classical three-dimensional case ($n=2$ and $\Omega=\{x_1^2+x_2^2\le 1\}\subset\R^2$). In~\cite{LachandPolygon}, the authors considered a natural subclass $D_M$ of convex bodies in $C_M$ that are a convex hull on the base $\Omega$ and a convex set $\omega_0$ of the plane $\{u=-M\}$. The side surface of the body lying in 
$D_M$ is a smooth developable surface. The body is defined by the set $\omega_0$. The authors showed that the set $\omega_0$ must be a regular polygon (or a segment) with center at the origin, the length of sides and the number of vertices of the polygon being defined by the height of the body $M$.

It is known that any optimal solution in the subclass $D_M$ can be nonoptimal in~$C_M$  (see~\cite{Buttazzo2009}). Numerical experiments show (see~\cite{LachardNumeric,Wachsmuth2014}) that optimal solutions in the class $C_M$ lead to a set $\omega_0$ being a regular polygon (or a segment) with center at the origin. But the side surface is nonsmooth and contains corners along some flat convex curves passing from the vertices of the polygon $\omega_0$ to the border of the base $\Omega$. For instance, for large heights $M$, the set $\omega_0$ is a segment, and the optimal solution seems to contain a corners along a convex curve in the vertical plane of symmetry $\{x_2=0\}$. In this paper, we consider a class of bodies which contains no additional corners. Hence, the side boundary must be a smooth developable surface. So we consider the following class:

\begin{defn}
	We say that a convex function $u\in C_M$ belongs to class $E_M$, if $u=\mathrm{conv}(\delta_\Omega, u_0)$, where $\epi u_0$ is the intersection of $\epi u$ and the vertical plane $\{x_2=0\}$, i.e.\ $u_0(x_1,0)=u(x_1,0)$ and $u_0(x_1,x_2)=\infty$ for $x_2\ne 0$..
\end{defn}

In this paper, we investigate in detail the class $E_M$ and the shapes of the optimal convex bodies in it. We shall obtain explicit formulas for the curve $u_0$ in the vertical plane $\{x_2=0\}$, and a family of solutions (depending on the height $M$) will be constructed for large enough $M$. It will be proved that each solution of the family provides a local minimum to the functional.

Usually in the calculus of variations, the term ``Maxwell stratum'' means the locus of points of intersection of different extremals with the same value of a functional. The extremals lose their optimality after the intersection with a Maxwell stratum. It is easy to find the Maxwell stratum if there is a symmetry. Then the Maxwell stratum appears naturally when an extremal intersects its own image. We consider height in Newton's aerodynamic problem as an analog of a functional in the calculus of variations and the generating lines of developable surfaces as extremals. Thus, if a convex body is symmetric with respect to a vertical plane and is smooth everywhere (except for points in the plane), then the generating lines of two symmetrical developable surfaces intersect at points in this plane. Having in mind the above analogy, we shall use the term ``Maxwell stratum'' for the intersection of the boundary of a symmetric convex body with its symmetry plane.

\medskip

Thus, let the domain of a convex function $u:\R^2\to\R$ be the unit circle $\Omega=\dom u=\{x_1^2+x_2^2\le 1\}$, and let
\[
u = \mathrm{conv}(\delta_\Omega,u_0) = \delta_\Omega \wedge u_0,
\]

\noindent where $\wedge$ is a short notation for the convex hull. The function $u_0(x_1,x_2)$ is equal to $\infty$ outside the segment $I=[(-1,0);(1,0)]$, while, on the segment, it is bounded by the numbers $-M$ and $0$, that is,
\[
\delta_I-M\le u_0\le \delta_I.
\]

Let us denote $v(p_1)=u_0^*(p_1,p_2)$ (the function $u_0^*(p_1,p_2)$ does not depend on  $p_2$). Consequently,

\[
|p_1|\le v\le |p_1|+M\quad\mbox{and}\quad u^*=\delta_\Omega^*\vee v=\max\{\sqrt{p_1^2+p_2^2},v(p_1)\}.
\]
\begin{thm}
	\label{thm:J_of_v}
	Let $\Omega=\dom u=\{x_1^2+x_2^2\le 1\}$, and let the shape of the convex body be $u=\delta_\Omega\wedge v^*$, where the convex function\footnote{Here and in what follows, the variable $p\in\R$ is a scalar.} $v(p):\R\to\R$ fulfills the conditions\footnote{Note that $\min \hat u=-M$ for an optimal solution~$\hat u$; hence $\hat v(0)=\hat u^*(0,0)=M>0$.} $|p|\le v(p)\le |p|+M$ and $v(0)>0$. Then the functional~\eqref{problem:start} has the form
	\[
	\mathcal{J}(u) = \int_{p_0^-}^{p_0^+}
	\Big[
	\frac{2\sqrt{v^2-p^2}(v')^2}{(1+v^2)^2} -
	\frac{pv'-v}{v(1+v^2)\sqrt{v^2-p^2}}
	\Big]dp,
	\]
	\noindent where $p_0^-=\inf\{p:v(p)>-p\}$ and $p_0^+=\sup\{p:v(p)>p\}$.
\end{thm}

To prove the theorem, we pass to the conjugate problem by using Theorem~\ref{thm:conjugate_problem}. Hence the main part of the proof is the calculation of the measure $F_0(dp_1,dp_2`|u^*)$. In the paper~\cite{LZ2018}, the measure was calculated under the additional assumption that the function $v$ is smooth. Below we shall prove that the optimal curve $v$ must be $C^1$-smooth everywhere except~0 (see Theorem~\ref{thm:convex_is_C_one} below). But this will need an explicit form of the functional $\mathcal{J}$ for nonsmooth functions $v$. Because of this, the assumption on the smoothness of $v$ in these circumstances is excessively restrictive and should be removed. Thus,

\begin{lemma}
	\label{lm:hess_mes}
	If $v(0)>0$, then the measure $F_0(dp_1,dp_2|u^*)$ is concentrated on the locally Lipschitzian curve\footnote{The curve $\gamma$ has two connected components for each maximal interval $v(p_1)>|p_1|$.} $\gamma = \{ p_1^2+p_2^2=v^2(p_1), p_2\ne 0\}$, and is given there by the formula
	\[
	F_0(dp_1,dp_2|u^*)|_\gamma=\frac12\Big[\frac{p_1v'-v}{p_2v}dp_1 - d\Big(\frac{p_2v'}{v}\Big)\Big].
	\]
\end{lemma}

\begin{proof}
	By definition, $u^*=\max\{s_\Omega,v\}$, where $s_\Omega(p_1,p_2)=\sqrt{p_1^2+p_2^2}$ and $v=v(p_1)$. Hence, if $\sqrt{p_1^2+p_2^2}>v(p_1)$ for $(p_1,p_2)\in\R^2$, then $F_0(dp_1,dp_2|u^*)=F_0(dp_1,dp_2|\delta^*_\Omega)$ in a neighbourhood of $(p_1,p_2)$. Similarly, if $\sqrt{p_1^2+p_2^2}<v(p_1)$, then $F_0(dp_1,dp_2|u^*)=F_0(dp_1,dp_2|v)$ in a neighbourhood of $(p_1,p_2)$.
	
	Since $v$ does not depend on $p_2$, $F_0(dp_1,dp_2|v)\equiv0$. Since the Hessian $\det (\delta^*_\Omega)''=0$ outside the origin, the measure $F_0(dp_1,dp_2|\delta^*_\Omega)$ is concentrated at the origin. Taking into account that $v(0)>0=\delta^*_\Omega(0)$, we see that the measure $F_0(dp_1,dp_2|u^*)$ is concentrated on the set of points $(p_1,p_2)$, such that $v(p_1)=\sqrt{p_1^2+p_2^2}$. These points generate the curve $\gamma$ and, perhaps, a set of points where $p_2=0$. We claim that the measure of this set relative to $F_0(dp_1,dp_2|u^*)$ equals zero. Indeed, at points of this set, the derivative $(\delta_\Omega^*)'$ and the sub-differential $\partial v(p_1)$ belong to the line $\{x_2=0\}$. Hence the area of their convex hull equals zero.
	
	Let us note that the curve $\gamma$ has the natural parametrization		
	\[
	p_2 = \pm\sqrt{v^2(p_1)-p_1^2}\ne 0,
	\]
	
	\noindent which is locally Lipschitzian, since the function $v(p_1)$ is convex and so it is locally Lipschitzian.	
	
	Now let us calculate the measure $F_0(dp_1,dp_2|u^*)$ on $\gamma$. For definiteness, we consider the part $\nu$ of the curve $\gamma$ lying in the upper half-plane $p_2>0$ (where $p_1\in[\alpha;\beta]$ for some $\alpha$, $\beta$) that does not contain points $p_2=0$. Let $\varepsilon>0$. Let us use the definition~\eqref{eq:schteiner_hessian_mesuare}. The area of the set
	\[
	\nu^\varepsilon=\{(p_1,p_2)+\varepsilon\partial u^*(p_1,p_2)|(p_1,p_2)\in\nu\}
	\]
	
	\begin{figure}[ht]
		\centering
		\includegraphics[width=0.4\textwidth]{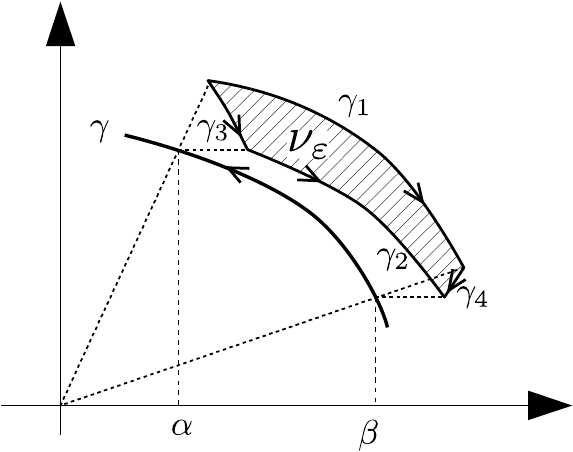}
		\caption{The structure of the set $\nu_\varepsilon$.}
		\label{fig:nu_epsilon}
	\end{figure}

	\noindent is a quadratic polynomial in $\varepsilon$, and, by definition, the coefficient at $\varepsilon^2$ is the value of the measure   $F_0(dp_1,dp_2|u^*)$ on $\nu$. The subdifferential of $u^*$ at points of $\gamma$ is given according to the formula of Dubovitzkij-Milutin as the convex hull of the subdifferentials $\sqrt{p_1^2+p_2^2}$ and $v(p_1)$. Therefore, the set $\nu_\varepsilon$ is bounded by the four Lipschitzian curves
	
	\begin{gather*}
	\gamma_1:\quad \Big(p_1+\varepsilon\frac{p_1}{v(p_1)},\sqrt{v(p_1)^2-p_1^2}+\varepsilon\frac{\sqrt{v(p_1)^2-p_1^2}}{v(p_1)}\,\Big),\\
	\gamma_2:\quad\Big(p_1 + \varepsilon\partial v(p_1),\sqrt{v(p_1)^2-p_1^2}\,\Big),\\
	\gamma_3:\quad\Big(\alpha + \varepsilon\big((1-t)\frac{\alpha}{v(\alpha)} + t v'(\alpha-0)\big),\sqrt{v(\alpha)^2-\alpha^2} + \varepsilon(1-t)\frac{\sqrt{v(\alpha)^2-\alpha^2}}{v(\alpha)}\,\Big),\\
	\gamma_4:\quad\Big(\beta + \varepsilon\big((1-t)\frac{\beta}{v(\beta)} + t v'(\beta+0)\big),\sqrt{v(\beta)^2-\beta^2} + \varepsilon (1-t)\frac{\sqrt{v(\beta)^2-\beta^2}}{v(\beta)}\,\Big).\\
	\end{gather*}
	
	\noindent Here $p_1\in[\alpha;\beta]$ and $t\in[0;1]$. Let us note that the boundary $\nu^\varepsilon$ passing counter-clockwise is $-\gamma_1+\gamma_2+\gamma_3-\gamma_4$ (as shown in Fig.~\ref{fig:nu_epsilon}). Since $v(p_1)>\sqrt{p_1^2+p_2^2}$ in a neighbourhood of the origin, the vector of the difference $\partial v(p_1) - \frac{(p_1,p_2)}{\sqrt{p_1^2+p_2^2}}$ is directed from the point $(p_1,p_2)$ to the origin at any point $(p_1,p_2)$ of the curve $\gamma$. 
	
	Let us show that the curves $\gamma_i$ do not intersect one another (except at the endpoints). The sets $u^*(p_1,p_2)+\varepsilon\partial u^*(p_1,p_2)$ do not intersect for different $(p_1,p_2)$ (as follows from convexity of $u^*$). Hence the curves $\gamma_1$ and $\gamma_2$ can intersect only for equal values of $p_1$. But this is also impossible, since, from the equality of the second coordinates, it follows that $p_2=\sqrt{v^2(p_1)-p_1^2}=0$, but this is prohibited by the choice of $\nu$. Similarly, we obtain this for the remaining pairs $\gamma_i$ and $\gamma_j$.
	
	Note that the curves $\gamma_1$, $\gamma_3$, and $\gamma_4$ are inscribed in the Lipschitzian parametrization. To inscribe in this parametrization $\gamma_2$, we shall act as follows: the function $v'(p_1)$ is monotonically increasing; hence its graph $\Gamma=\{(p_1,q_1):q_1\in\partial v(p_1)\}$ has a finite length, which we choose as a parameter. So the curve $\Gamma$ will be given by two Lipschitzian functions $p_1=P(l)$ and $q_1=Q(l)$. Then the curve $\gamma_2$ takes the form	
	\[
	\gamma_2:\quad p_1=P(l)+\varepsilon Q(l),\ p_2 = \sqrt{v^2(P(l))-P^2(l)}.
	\]
	\noindent This parametrization is Lipschitzian, because the function $v(p_1)$ is convex.
	
	Hence the set $\nu_\varepsilon$ is bounded by a Lipschitzian curve and its area can be calculated by Green's formula. Taking the counter-clockwise direction, we obtain
	\[
	L^2(\nu^\varepsilon) = \int_{\gamma_1} p_2dp_1 - \int_{\gamma_2} p_2dp_1 - \int_{\gamma_3} p_2dp_1 + \int_{\gamma_4} p_2dp_1.
	\]
	\noindent The second summand does not affect the coefficient of $\varepsilon^2$. Hence this coefficient has the form
	\[
	\frac12\frac{d^2}{d\varepsilon^2}L^2(\nu^\varepsilon) =
	\int_\alpha^\beta \frac{\sqrt{v^2(p_1)-p_1^2}}{v(p_1)}d\Big(\frac{p_1}{v(p_1)}\Big) +
	\frac12\Big(\frac{v(p_1)v'(p_1)-p_1}{v^2(p_1)}\sqrt{v^2(p_1)-p_1^2}\Big)
	\Big|_{\alpha-0}^{\beta+0}.
	\]
	The origin lies inside the domain bounded by the curve $\gamma$. Choosing the counter-clockwise direction, we see that the measure $F_0(dp_1,dp_2|u^*)$ on the part $\gamma$ with $p_2>0$ takes the form (with regard to $p_2=\sqrt{v^2(p_1)-p_1^2}$)
	\[
	F_0(dp_1,dp_2|du^*)|_{\gamma\cap\{p_2>0\}} =-\frac{p_2}{v(p_1)}d\Big(\frac{p_1}{v(p_1)}\Big) - \frac12d\Big(\frac{v(p_1)v'(p_1)-p_1}{v^2(p_1)}p_2\Big).
	\]
	\noindent Similar calculations show that, on the lower part of $\gamma$, where $p_2<0$, the measure $F_0(dp|u^*)$ is given by the same formula. Taking into account the fact that $vv'dp_1\stackrel{\gamma}{=}p_1dp_1 + p_2dp_2$, we obtain
	
	\begin{multline*}
	F_0(dp_1,dp_2|du^*)|_\gamma = \frac{-p_2dp_1}{v^2} + \frac{p_1p_2dv}{v^3} - \frac12d\Big(\frac{p_2v'}{v}\Big) + \frac12d\Big(\frac{p_1p_2}{v^2}\Big)=\\
	=\frac12\Big[\frac{p_1dp_2-p_2dp_1}{v^2} - d\Big(\frac{p_2v'}{v}\Big)\Big]=
	\frac12\Big[\frac{p_1v'-v}{p_2v}dp_1 - d\Big(\frac{p_2v'}{v}\Big)\Big].
	\end{multline*}
\end{proof}

\begin{proof}[Proof of Theorem~\ref{thm:J_of_v}.]
	By Theorem~\ref{thm:conjugate_problem} the resistance of the body constructed by the Maxwell stratum $v^*$ is given by the formula
	
	\[
	\mathcal{J}(u) = \int_\gamma \frac{1}{1+p_1^2+p_2^2} F_0(dp_1,dp_2|u^*)|_\gamma
	\]
	
	\noindent or, with regard to $v(p_1)\stackrel{\gamma}{=}\sqrt{p_1^2+p_2^2}$ and Lemma~\ref{lm:hess_mes}, we obtain\footnote{Since $\sqrt{p_1^2+p_2^2}\le v\le \sqrt{p_1^2+p_2^2}+M$, the curve $\gamma$ consists exactly of two symmetric connected components: one for $p_2>0$ and the other for $p_2<0$.}
	\[
	\mathcal{J}(u) = \int_{p_0^-}^{p_0^+}
	\frac{1}{1+v^2}\Big[-\frac{p_1v'-v}{\sqrt{v^2-p_1^2}v}dp_1 + d\Big(\frac{v'\sqrt{v^2-p_1^2}}{v}\Big)\Big].
	\]
	\noindent Let us find~$p_0^-$ and $p_0^+$. If $v(q)=-q$ for some number~$q$, then $v(r)=-r$ for all $r\le q$, since $|p_1|\le v(p_1)\le |p_1|+M$ and $v$ is convex. Hence $p_0^-$ is the greatest solution to the equation $v(p_1)=- p_1$ (if it exists) or $p_0^-=-\infty$ (in the opposite case). Similarly, we find~$p_0^+$. Integrating the last summand in the previous integral by parts, we obtain
	\[
	\mathcal{J}(u) = \int_{p_0^-}^{p_0^+}
	\Big[
	\frac{2\sqrt{v^2-p_1^2}(v')^2}{(1+v^2)^2} - \frac{p_1v'-v}{v(1+v^2)\sqrt{v^2-p_1^2}}
	\Big]dp_1.
	\]
	\noindent The terminal parts vanish, since $v'\sqrt{v^2-p_1^2}/v\to 0$ both as $p_1\to p_0^-+0$ and as $p_1\to p_0^+-0$. Indeed, $v'\in[-1;1]$ and if $p_0^-\ne -\infty$, then $v(p_0^-)=-p_0^-$, and if $p_0^-=-\infty$, then $v\to +\infty$ as $p\to -\infty$ (similarly, for $p_0^+$).
	
\end{proof}

In what follows, we shall drop the subscript for $p_1$ writing for simplicity $p\in\R$, as was done in Theorem~\ref{thm:J_of_v}. Thus, we have the problem
\[
\int_{p_0^-}^{p_0^+}\left[
\frac{2\sqrt{v^2-p^2}(v')^2}{(1+v^2)^2} - \frac{pv'-v}{v(1+v^2)\sqrt{v^2-p^2}}
\right]\,dp\to\min_v.
\]

\noindent The minimum must be found in the class of convex functions $v:\R\to\R$, $v(0)=M>0$, satisfying the inequalities
\begin{equation}
\label{eq:func_v_ineq}
|p|\le v(p) \le |p|+M,
\end{equation}
\noindent and the boundary values $p_0^-$ and $p_0^+$ are the extreme solutions of the equations $v=\pm p$ (respectively), or $p_0^-=-\infty$ or $p_0^+=\infty$, if there are no solutions of the corresponding equations. Since $v(0)=M>0$, we have $p_0^-<0$ and $p_0^+>0$.

\begin{remark}
	The problem is invariant under the substitution of $p$ for $-p$. Therefore, it is natural to consider the problem on the interval $[0;p_0]$. If we find a solution in the problem on the interval $[0;p_0]$, then we must reflect it symmetrically to $[-p_0;0]$. But it is necessary to check that the solution on $[0;p_0]$ has a nonnegative derivative in $0$ (otherwise, the symmetric prolongation gives a nonconvex function). In what follows, we shall verify that the constructed solutions on $[0;p_0]$ indeed have a positive derivative at $0$. Therefore, the solution on $[p_0^-;p_0^+]$ is symmetric. 
\end{remark}

Hence we have the following key problem: To find a convex function $v:\R\to\R$ that minimizes the following functional:

\begin{equation}
\label{eq:maxwell_integral_value}
\boxed{
	\begin{gathered}
	J(v)=\int_0^{p_0} \left[
	\frac{2\sqrt{v^2-p^2}(v')^2}{(1+v^2)^2} - \frac{pv'-v}{v(1+v^2)\sqrt{v^2-p^2}}
	\right]\,dp\to\min_v\\
	v(0)=M;\qquad p_0=\sup\{p:v(p)>p\};\qquad 0\le p\le v(p)\le p+M.
	\end{gathered}
}
\end{equation}

\noindent We shall frequently write the condition $0\le p\le v\le p+M$ in the form $(p,v)\in U$.

Let us formulate in local terms the condition that the convex curve $v(p)$ cannot be improved. Let us replace $v$ on the interval $[\alpha;\beta]\subset[0;p_0]$ by a new curve $\tilde v$. Convexity of of the modified curve is equivalent to the following: (i)~convexity of $\tilde v$ on $[\alpha;\beta]$, (ii) continuity of the junction at the end points $\tilde v(\alpha)=v(\alpha)$ and $\tilde v(\beta)=v(\beta)$, and (iii) conservation of the monotonicity of the derivative at the points $\alpha$ and $\beta$, i.e., $\tilde v'(\alpha+0)\ge v'(\alpha-0)$ and $\tilde v'(\beta-0)\le v'(\beta+0)$. Besides, it is necessary to require that the new curve remain in the domain $U$. So the curve~$v$ cannot be improved on~$[\alpha;\beta]$ if any curve $\tilde v$ satisfying all the above-mentioned conditions gives larger values to the integral~\eqref{eq:maxwell_integral_value} on~$[\alpha;\beta]$.

\section{Smoothness of the optimal solution}

The role of the second derivative of the convex function can be played, generally speaking, by any nonnegative measure. Hence, formally speaking, to remove the restriction $v''\ge 0$ in problem~\eqref{eq:maxwell_integral_value}, we need to use Pontryagin's Maximum Principle for an impulse-type control. For instance, if we put $v'=w$ and $w'=\theta\ge 0$, then the obtained trajectory $(v(p),w(p))$ will be discontinuous at points where the measure~$\theta$ has atoms. In this section, we shall prove that the optimal solution must belong to the class $C^1$, i.e., $v,w\in C$. In this case, it may still appear that the measure~$\theta$ has a purely singular component. But it follows from Pontryagin's Maximum Principle that the adjoint variables are $C^1$-smooth, which essentially simplifies the investigation of the problem.

So let us formulate the problem in general form. Let $U\subset\R^2$ be a convex set, $\Int U \neq \emptyset$. Consider the problem  

\begin{equation}
\label{eq:general_convex_var_problem}
\begin{gathered}
I(x)=\int_{t_0}^{t_1} f(t,x,\dot x)\,dt\to\min,\\
x\mbox{ is a convex function},\quad (t,x(t))\in U\mbox{ for all }t\in[t_0;t_1],\\ 
\quad x(t_0)=x_0,\ x(t_1)=x_1,\ \ \dot x(t_0+0)\ge s_0,\ \dot x(t_1-0)\le s_1,
\end{gathered}
\end{equation}

\noindent where $x_{0,1}$, $t_{0,1}$, and $s_{0,1}$ are given numbers.

It is important to require that the strict inequalities $t_0<t_1$ and $s_0<s_1$ hold. Indeed, if $t_0=t_1$, then the interval of integration reduces to a point, and if $s_0=s_1$, then the only possible admissible curve is a segment of a straight line and only in the case $s_0=s_1=(x_1-x_0)/(t_1-t_0)$.

Let us note that the Euler--Lagrange equation needs not to be fulfilled, because any small variation can break the key convexity condition.

\begin{thm}
	\label{thm:convex_is_C_one}
	Let $f\in C^2(\R^3\to\R)$, and let $\hat x(t)$ be an optimal solution of problem~\eqref{eq:general_convex_var_problem}. Suppose that $(t,x(t))\in\Int U$ for all $t\in(t_0,t_1)$. Suppose that the strong Legendre condition
	\[
	f_{\dot x\dot x}(t,x,\dot x)>0
	\]
	
	\noindent is fulfilled for all points $(t,x,\dot x)$ of the graph $\Gamma=\{(t,x,\dot x):t\in[t_0,t_1], x=\hat x(t), \dot x\in\partial\hat x(t)\}\subset\R^3$. Then $\hat x\in C^1[t_0;t_1]$.
	
\end{thm}

\begin{proof}
	
	It is obvious that $\dot{\hat x}$ is a monotonically increasing function. Thus, there exists right and left derivatives at~$t_0$ and~$t_1$, respectively. They are finite, since 
	\[
	-\infty<s_0\le \dot x(t_0+0)\le \dot x(t_1-0)\le s_1<\infty. 
	\]
	\noindent Moreover, the function~$x$ is convex, so we only need to prove that, the derivative~$\dot x(\tau)$ exists for any~$\tau\in(t_0,t_1)$, since the derivatives of convex functions are always monotonic (see~\cite{Rockafellar}).
	
	We prove the theorem by contradiction. Suppose that there exists a moment $\tau\in(t_0,t_1)$, such that $\dot{\hat x}(\tau+0)-\dot{\hat x}(\tau-0)=\sigma>0$. 
	
	At first, let us choose a neighborhood $\Pi$ where all needed points will lie. We can do it in the following way. The graph $\Gamma$ contains the (vertical) segment that joins the points $(\tau,\hat x(\tau),\dot{\hat x}(\tau-0))$ and $(\tau,\hat x(\tau),\dot{\hat x}(\tau+0))$. Under the conditions of the theorem, the inequality $f_{\dot x\dot x}(t,x,\dot x)>0$ is fulfilled for all points on the segment. Hence there exist a number $\gamma>0$ and a rectangular neighbourhood of the segment $\Pi=[\tau-\varepsilon_0;\tau+\varepsilon_0]\times[\hat x(\tau)-\varepsilon_0;\hat x(\tau)+\varepsilon_0]\times[\dot{\hat x}(\tau-0)-\varepsilon_0,\dot{\hat x}(\tau+0)+\varepsilon_0]$ such that $f_{\dot x\dot x}(t,x,\dot x)\ge\gamma>0$ for all points $(t,x,\dot x)\in\Pi\subset\R^3$. In addition, the number $\varepsilon_0>0$ can be chosen in such a way that the set $[\tau-\varepsilon_0;\tau+\varepsilon_0]\times[\hat x(\tau)-\varepsilon_0;\hat x(\tau)+\varepsilon_0]$ belongs to $U$, since $(\tau,\hat x(\tau))\in\Int U$ by the conditions of the theorem.
	
	Let us construct a variation of $\hat x$ having the form of a cut-off function. Let $\varepsilon>0$. We set
	\[
	x_{\varepsilon}(t) = 
	\begin{cases}
	\frac{1}{2\varepsilon}\big((t-\tau+\varepsilon)\hat x(\tau+\varepsilon) + (\tau+\varepsilon-t)\hat x(\tau-\varepsilon)\big)&\mbox{if } t\in[\tau-\varepsilon;\tau+\varepsilon];\\
	\hat x(t)&\mbox{otherwise}.\\
	\end{cases}
	\]
	\noindent In other words, the convex (continuous) function $x_\varepsilon$ is obtained from $\hat x$ by replacing its values on the interval $[\tau-\varepsilon;\tau+\varepsilon]$ by the values of the corresponding linear function.
	
	If $\varepsilon$ is sufficiently small, then the parallelepiped $\Pi$ contains points $(t,\hat x(t),\dot{\hat x}(t))$ for $t=\tau\pm\varepsilon$. Consequently, the graph of $x_\varepsilon$ belongs to $\Pi$. Hence the integral $I(x_\varepsilon)$ is well defined, and
	\[
	I(x_\varepsilon) - I(\hat x) = 
	\int_{\tau-\varepsilon}^{\tau+\varepsilon}
	\big[
	f(t,x_\varepsilon(t),s_\varepsilon) - \hat f(t)
	\big]\,dt,
	\]
	
	\noindent where $s_\varepsilon = (x(\tau+\varepsilon)-x(\tau-\varepsilon))/2\varepsilon$. 
	
	Let us remark that if the number $\varepsilon$ is sufficiently small, then points $(t,\hat x(t),s_\varepsilon)$ and $(t,x_\varepsilon(t),s_\varepsilon)$ belong to $\Pi$ for all $t\in[\tau-\varepsilon;\tau+\varepsilon]$. Indeed, since $s_\varepsilon\in[\dot{\hat x}(\tau-\varepsilon+0);\dot{\hat x}(\tau-\varepsilon-0)]$, it follows that, for $\varepsilon\to0$, we have $\dot{\hat x}(\tau\mp\varepsilon\pm0)\to\dot{\hat x}(\tau\pm0)$ in view of monotonicity of $\dot{\hat x}(t)$. Hence, for a sufficiently small $\varepsilon$, we have $s_\varepsilon\in[\dot{\hat x}(\tau-0)-\varepsilon_0,\dot{\hat x}(\tau+0)+\varepsilon_0]$. The inclusion $t\in[\tau-\varepsilon_0;\tau+\varepsilon_0]$ is fulfilled trivially for $\varepsilon<\varepsilon_0$, and the inclusion $x(t)\in [\hat x(\tau)-\varepsilon_0;\hat x(\tau)+\varepsilon_0]$ follows from the continuity of~$\hat x$.
	
	According to what has been said above, it follows that
	\[
	|f(t,x_\varepsilon(t),s_\varepsilon)-f(t,\hat x(t),s_\varepsilon)| \le 
	\sup_{\Pi}|f_x|\, |x_\varepsilon(t)-\hat x(t)| =
	O(\varepsilon),
	\]
	
	\noindent because $|x_\varepsilon(t)-\hat x(t)|=O(\varepsilon)$ in view of the boundedness of the derivative  $|\dot{\hat x}|$. Hence
	\[
	I(x_\varepsilon) - I(\hat x) = 
	\int_{\tau-\varepsilon}^{\tau+\varepsilon}
	\big[
	f(t,\hat x(t),s_\varepsilon) - \hat f(t)
	\big]\,dt + O(\varepsilon^2).
	\]
	
	Now let us use the strong Legendre condition (i.e., the strict convexity of $f$ relative to $\dot x$)
	\[
	f(t,\hat x(t),\dot{\hat x}(t))\ge
	f(t,\hat x(t),s_\varepsilon) +
	f_{\dot x}(t,\hat x(t),s_\varepsilon)(\dot{\hat x}(t)-s_\varepsilon) +
	\gamma (\dot{\hat x}(t)-s_\varepsilon)^2.
	\]
	
	\noindent The inequality is valid, because the points $(t,\hat x(t),\dot{\hat x}(t))$ and $(t,\hat x(t),s_\varepsilon)$ belong to $\Pi$ for $t\in[\tau-\varepsilon;\tau+\varepsilon]$. Hence
	\[
	I(x_\varepsilon)-I(\hat x) \le -\gamma\int_{\tau-\varepsilon}^{\tau+\varepsilon}
	(\dot{\hat x}(t)-s_\varepsilon)^2\,dt +
	\big|
	\int_{\tau-\varepsilon}^{\tau+\varepsilon}
	f_{\dot x}(t,\hat x(t),s_\varepsilon)(\dot{\hat x}(t)-s_\varepsilon)
	\,dt
	\big|+
	O(\varepsilon^2).
	\]
	
	Since
	\[
	|f_{\dot x}(t,\hat x(t),s_\varepsilon)-f_{\dot x}(\tau,\hat x(\tau),s_\varepsilon)|\le
	\sup_\Pi|f_{t\dot x}||t-\tau| + \sup_\Pi|f_{x\dot x}||\hat x(t)-\hat x(\tau)| =		
	O(\varepsilon),
	\]	
	\noindent we obtain
	\[
	\big|
	\int_{\tau-\varepsilon}^{\tau+\varepsilon}
	f_{\dot x}(t,\hat x(t),s_\varepsilon)(\dot{\hat x}(t)-s_\varepsilon)
	\big| \le
	O(\varepsilon^2) + |f(\tau,\hat x(\tau),s_\varepsilon)|
	\big|
	\int_{\tau-\varepsilon}^{\tau+\varepsilon}
	(\dot{\hat x}(t)-s_\varepsilon)\,dt
	\big|.
	\]
	
	\noindent Let us note that the last integral is zero. So
	
	\[
	I(x_\varepsilon)-I(\hat x) \le -\gamma\int_{\tau-\varepsilon}^{\tau+\varepsilon}
	(\dot{\hat x}(t)-s_\varepsilon)^2\,dt + O(\varepsilon^2).
	\]
	
	We now consider the location of the point $s_\varepsilon$ relative to the interval $[\dot{\hat x}(\tau-0);\dot{\hat x}(\tau+0)]$ of length 
	$\sigma>0$. Let $s_\varepsilon$ lie not higher than the midpoint of the interval (the second case is considered similarly). Since $\dot{\hat x}$ is increasing for $t\ge\tau$, we have $\dot{\hat x}(t)-s_\varepsilon\ge\sigma/2$; thus,
	
	\[
	\int_{\tau-\varepsilon}^{\tau+\varepsilon}
	(\dot{\hat x}(t)-s_\varepsilon)^2\,dt \ge
	\int_{\tau}^{\tau+\varepsilon}(\dot{\hat x}(t)-s_\varepsilon)^2\,dt \ge
	\left(\frac{\sigma}{2}\right)^2\varepsilon
	\]
	
	\noindent and
	
	\[
	I(x_\varepsilon)-I(\hat x) \le -\gamma\left(\frac{\sigma}{2}\right)^2\varepsilon + O(\varepsilon^2),
	\]
	
	\noindent which contradicts the optimality of the trajectory $\hat x$.
\end{proof}

\section{Pontryagin's Maximum Principle for the key problem}

Let us check that the conditions of Theorem~\ref{thm:convex_is_C_one} for problem~\eqref{eq:maxwell_integral_value} are fulfilled. Since
\begin{equation}
\label{eq:f}
f(p,v,v') = \frac{2\sqrt{v^2-p^2}(v')^2}{(1+v^2)^2} - \frac{pv'-v}{v(1+v^2)\sqrt{v^2-p^2}},
\end{equation}

\noindent we have~$f_{v'v'}>0$ for all points in the interior of the set~$U=\{(p,v):0\le p\le v\le p+M\}$. Thus, any optimal solution must be at least $C^1$-smooth in $\Int U$.

We use the following main idea to construct the optimal synthesis in problem~\eqref{eq:maxwell_integral_value}. We apply Pontryagin's Maximum Principle for the key problem and use it to obtain the field of extremals, which cover a subdomain in~$U$. Then the classical Legendre construction allows us to use the solutions of a Riccati equation to show that any sufficiently close $PC^2$-curve\footnote{A function $u$ belongs to class $PC^2$ if its first derivative $u'$ is $C^1$, and its second derivative $u''$ is piecewise continuous.} gives larger values to the functional~$J$ than the corresponding extremal from the field.

So we consider the following optimal control problem:
\begin{equation}
\label{problem:main_maxwell}
\begin{gathered}
J(v)=\int_0^{p_0} f(p,v,v')\,dp\to\min;\\
v' = w ;\quad w'=\theta\ge 0;\\
v(0)=M;\quad v(p_0)=p_0<\infty,
\end{gathered}
\end{equation}
\noindent where the integrand is given by formula~\eqref{eq:f}. The variable~$p$ plays the role of time, $v$ and $w$ are the phase variables, and $\theta$ is the control.

For an optimal solution, we have $v(0)=M$, since $\min u=-M$, but, generally speaking, the condition~$v(p_0)=p_0$ may not be fulfilled on the optimal solution. Also we have dropped the condition~$(p,v(p))\in U$. Nonetheless, we shall construct an optimal synthesis in a subdomain of~$U$ (with different~$p_0$ and~$M$) using exactly these conditions (see Sec.~\ref{sec:field_of_extremals} below). After that, we shall prove that every constructed solution is a local minimum (see Theorem~\ref{thm:second_var} below).

Let us remark that~$f$ has a singularity at the right endpoint~$v(p_0)=p_0$. Consequently, Pontryagin's Maximum Principle cannot be applied directly. In other words, formally speaking, Pontryagin's Maximum Principle may not be a necessary optimality condition. Nonetheless, it appears that it is a sufficient optimality condition in the key problem. Namely, in Sec.~\ref{sec:field_of_extremals}, we construct a field of extremals that satisfy the equations of Pontryagin's Maximum Principle and prove their local optimality in Sec.~\ref{sec:second_var}.

So let us write down Pontryagin's Maximum Principle for~$PC^2$ extremals. It states that there exist a number~$\lambda_0\ge 0$ and functions $\varphi(p)$ and $\psi(p)$ (not equal to 0 at the same time) such that the equations $\varphi'=-H'_v$, $\psi'=-H_w$ are held for
\begin{equation}
\label{eq:PMP}
H=-\lambda_0 f(p,v,w) + \varphi w + \psi \theta.
\end{equation}

The function $\psi(p)$ must be nonpositive and must vanish on the support of $\theta$, $\psi|_{\mathrm{supp}\,\theta}=0$. The orthogonality condition for the adjoint variables are
\[
\psi(0)=0,\quad \psi(p_0)=0\quad\mbox{and}\quad \varphi(p_0)=H(p_0).
\]

Let us remark that the function $f$ has a singularity at the point $p_0$, since $v(p_0)=p_0$ and the denominator of the second fraction in~\eqref{eq:f} becomes zero. So it may appear that, on the extremals, $H\to\infty$ as $p\to p_0-0$. Hence, the last orthogonality condition must be understood as follows: $\varphi-H\to 0$ as~$p\to p_0-0$.

Let us show that $\lambda_0\ne 0$. Indeed, if $\lambda=0$, then $\varphi' = 0$ and $\psi' = -\varphi$. Hence $\psi(p)$ is a linear function and $\psi(0)=\psi(p_0)=0$. That is, $\psi=\varphi=0$ and all the adjoint variables are zero, which is forbidden.

We take $\lambda_0=1$. Therefore,
\[
\boxed{\varphi' = f_v;\quad \psi' = f_{v'} - \varphi}.
\]
\noindent By Theorem~\ref{thm:convex_is_C_one}, any optimal solution is $C^1$-smooth on $\Int U$, i.e., for~$p\in(0;p_0)$. Hence we see that the functions $\varphi$ and $\psi$ are also $C^1$-smooth on~$(0;p_0)$.

The control $\theta$ is defined by the adjoint variable $\psi$, which can be found by Cauchy's formula: for any $p_1$ and $p_2$, we have
\begin{equation}
\label{eq:psi_eq_int}
\psi(p_2)= \psi(p_1) + \psi'(p_1)(p_2-p_1) + 
\int_{p_1}^{p_2} (p_2-p) \left(\frac{d}{dp}f_{v'}-f_v\right)\,dp.
\end{equation}

The Maximum Principle implies that if $\psi<0$ on a segment, then $\theta=0$ and the trajectory $v$ is an affine function on this segment. The Maximum Principle also admits singular arcs, when $\psi\equiv0$ in an interval. We claim that the trajectory must be singular in a neighbourhood of $p_0$. This follows from the following proposition.

\begin{prop}
	\label{prop:v1_should_be_zero}
	If $v'(p_0)=1-a$ for $a>0$\footnote{The case $a<0$ needs not to be considered, since, in this case, the condition $v>p$ fails in a left neighbourhood of $p_0$.}, then the trajectory $v$ does not satisfy Pontryagin's Maximum Principle. 
\end{prop}

Indeed, the proposition implies that if the trajectory $v$ is nonsingular on an interval $[p_0-\Delta p;p_0]$, $\Delta p>0$, then $v'=\const$ on this interval, and since $v\ge p$, we have $v'(p)=1$ and $v(p)=p$ for $p\in[p_0-\Delta p;p_0]$. Therefore, in this case, $p_0$ does not fulfill the condition $p_0=\sup\{p:v(p)>p\}$ and it can be reduced.

\begin{proof}[Proof of proposition~\ref{prop:v1_should_be_zero}]
	Let us estimate the asymptotic of the function $\varphi(p)$ as $p\to p_0-0$. We use the orthogonality conditions: $\alpha=H-\varphi\to 0$ as $p\to p_0-0$. Since $\psi\theta\equiv0$, it follows that, on the trajectory, $H=-f+\varphi v'$. While expressing $\varphi$ in terms of $\alpha$, we find 
	\[
	\varphi = \frac{f+\alpha}{v'-1} = -\frac{1}{(1+p_0^2)\sqrt{2p_0}\sqrt{a(p_0-p)}} + o(1).
	\]
	\noindent The function $f_{v'}$ on the trajectory has the same asymptotic as $p\to p_0-0$:
	\[
	f_{v'} = -\frac{1}{(1+p_0^2)\sqrt{2p_0}\sqrt{a(p_0-p)}} + o(1).
	\]
	\noindent Therefore, $\psi'(p)=f_{v'}-\varphi=o(1)$ and, consequently, $\psi'(p_0)=0$.
	
	Let us calculate $\psi''$:
	\[
	\psi'' = \frac{d}{dp}f_{v'} -f_v = f_{v'v'}v'' - (f_v-f_{pv'}-f_{vv'}v').
	\]
	\noindent The first term is nonnegative, since $f_{v'v'}>0$ and $v''\ge 0$. The second term has the form
	\[
	f_v-f_{pv'}-f_{vv'}v'=f_{v'v'}\left(-\frac14\frac{(v'-1)^2}{v-p} - \frac14\frac{(v'+1)^2}{v+p} + \frac{2vv'^2}{1+v^2}\right).
	\]
	\noindent Here the second and third terms on the right-hand side are bounded and the first term tends to $-\infty$ as $p\to p_0-0$ and has order $(p-p_0)^{-1}$. Since $f_{v'v'}\ge 0$ and $f_{v'v'}=O(\sqrt{p_0-p})$, we obtain $\psi''\to+\infty$ as $p\to p_0-0$. Therefore, $\psi''>0$ in a left neighbourhood of the point $p_0$. Taking into account the fact that $\psi(p_0)=\psi'(p_0)=0$, we obtain $\psi(p)>0$. This contradicts Pontryagin's Maximum Principle.	
\end{proof}

Hence we seek a trajectory $v$ under the condition $v'(p_0)=1$. A nonsingular trajectory cannot meet this condition. Hence the trajectory must be singular in a neighbourhood of $p_0$ (that is, $\psi\equiv0$ in a neighbourhood of $p_0$). To find the singular control on the interval $[p_1;p_2]$, one must differentiate twice the condition $\psi\equiv0$. The result will be the classical Euler--Lagrange equation $df_{v'}/dp-f_v=0$, which must be fulfilled on singular arcs. The singular control~$\theta=v''$ can be obtained from this equation. A direct computation gives the following equation on singular arcs, which holds for a.e.~$p\in[p_1;p_2]$,

\begin{equation}
\label{eq:main_Lagrange_eq}
\boxed{v''= \frac{f_v-f_{pv'}-v'f_{vv'}}{f_{v'v'}}=-\frac14\frac{(v'-1)^2}{v-p} - \frac14\frac{(v'+1)^2}{v+p} + \frac{2vv'^2}{1+v^2}.}
\end{equation}

The equation has a singularity in a neighbourhood at $v(p_0)=p_0$, and a further investigation is needed. This will be done in Sec.~\ref{sec:analitic_expansion}.

We need to complexify  equation~\eqref{eq:main_Lagrange_eq}. It would be useful to have one of the famous fully understood types of complex equations. Equation~\eqref{eq:main_Lagrange_eq} is of type most closely resemble with that of Painlev\'e. Painlev\'e equations do not have moving ramification points. The Painlev\'e equations of second order are classified. In this case, the right-hand side should be a quadratic polynomial in the first derivative $v'$ and the term $v'^2$ must have a coefficient of the form $\sum_k a_k/(v-v_k)$, which correspond exactly to our case. But, as is the case in Painleve\'e equations, the corresponding residue at $v_k$ must be equal to $1+1/N$, where $N$ belongs to $\mathbb{Z}$. Equation~\eqref{eq:main_Lagrange_eq} almost has the described structure, but the residue equals -1/4, which is not equal to $1+1/N$. Hence, equation~\eqref{eq:main_Lagrange_eq} does have moving ramification points and is not of Painlev\'e type. The line $\{v=p\}$ consists of ramification points. Nevertheless, the methods used in the investigation of Painlev\'e equations~\cite[p.149-160]{Golubev} appear to be helpful in our case as well.

\medskip

Before investigating singular extremals, let us find conditions on a nonsingular interval. The interval $[p_1;p_2]$ is a maximal nonsingular interval if the following equation is fulfilled:

\begin{equation}
\label{eq:psi_p_not_sing}
\psi(p)<0\mbox{ for } p\in(p_1;p_2)\quad\mbox{and}\quad\psi(p_1)=\psi(p_2)=0.
\end{equation}

\noindent Indeed, if $p_1>0$ and $p_2<p_0$, then $\psi(p_1)=\psi(p_2)=0$ in view of the maximality of the nonsingular interval. And if $p_1=0$ or $p_2=p_0$, then $\psi(0)=0$ and $\psi(p_0)=0$ in view of the orthogonality condition.

In Sec.~\ref{sec:field_of_extremals}, we shall cover a sub-region in $U$ by extremals with one singular and one nonsingular intervals, and, in Sec.~\ref{sec:second_var}, we will prove the local optimality of the obtained extremals (see Theorem~\ref{thm:second_var} below).

\section{Basic properties of singular extremals}
\label{sec:analitic_expansion}

The Euler--Lagrange equation~\eqref{eq:main_Lagrange_eq} for singular extremals has a singularity at the right end, since the denominator of the first fraction vanishes at the point~$v(p_0)=p_0$. Thus, the behavior of its solutions in a neighborhood of this point deserves an accurate study. Unfortunately, the authors do not know any literature where equations with singularities of this type are studied. Actually, in this section, we are compelled to develop some basic ODE results for this type of singularities. 

We start with representing equation~\eqref{eq:main_Lagrange_eq} in general form. Let us carry out the time shift~$t=p-p_0$ and change~$x=v-p$. Then (for a.e.~$t$) we have
\[
\ddot x = -\frac14 \frac{\dot x^2}{x} - \frac14\frac{(\dot x+2)^2}{(x+2t+2p_0)} +
\frac{2(x+t+p_0)(\dot x+1)^2}{1+(x+t+p_0)^2},
\]

\noindent where~$x$ satisfies the initial condition~$x(0)=0$. Let us pick out the main term and rewrite the previous equation in the following general form:
\begin{equation}
\label{eq:analytic_with_singularity}
\ddot x = \lambda \frac{\dot x^2}{x} + g(t,x,\dot x)
\mbox{ for a.e.\ }t,
\qquad\mbox{and}\qquad x(0) = 0.
\end{equation}

Note that, for a given function~$x(t)$, the right-hand side is, possibly, not well defined in any punctured neighborhood of the point~$t=0$. Moreover, if we multiply the equation by~$x$, we get an equation with trivial solution~$x\equiv0$ (or $v=p$). But, in fact, we are interested in a solution~$v$, which is convex and satisfies the inequality~$v>p$ for~$p<p_0$. Thus, $v'<1$ for~$p<p_0$. So we are searching for a solution of equation~\eqref{eq:analytic_with_singularity} such that~$\dot x\ne 0$ in a left punctured neighborhood of~$t=0$.

\begin{thm}
	\label{thm:analytic_solution_extsts_and_unique}
	Assume that~$0<|\lambda|<3/8$, $g(0,0,0)\ne 0$, and the function~$g$ is continuous together with its partial derivatives~$g_x$ and~$g_{\dot x}$ in a neighborhood of the point~$(0,0,0)$. Then there exists a number~$\tau>0$ such that equation~\eqref{eq:analytic_with_singularity} has a unique $C^2$-solution on~$t\in[-\tau;\tau]$ with the initial condition~$x(0)=0$ and the property that $\dot x \ne 0$ in a left punctured neighborhood of~$t=0$. This solution satisfies\footnote{The similar result holds for a solution~$\tilde x$ with the property that~$\dot{\tilde x}\ne 0$ in a right punctured neighborhood of~$t=0$. Moreover, the solutions~$x$ and~$\tilde x$ must coinside, since both do not vanishes in a punctured two-sided neighborhood of~$t=0$ as~$\ddot x(0)=\ddot{\tilde x}(0)\ne 0$.}
	\[
	\dot x(0)=0\quad\mbox{and}\quad\ddot x(0)= \frac{g(0,0,0)}{1-2\lambda}.
	\]
	\noindent Additionally, if the function~$g$ is analytic, then this solution is analytic too.
\end{thm}

Equation~\eqref{eq:main_Lagrange_eq} satisfies all the above conditions, since we have
\[
\lambda = -\frac14,\quad 
g = - \frac14\frac{(\dot x+2)^2}{(x+2t+2p_0)} + \frac{2(x+t+p_0)(\dot x+1)^2}{1+(x+t+p_0)^2},\quad
g(0,0,0)=\frac{3 p_0^2-1}{2 p_0(1+p_0^2)},
\]
\noindent and~$g(0,0,0)\ne 0$ for $p_0\ne 1/\sqrt{3}$. Remark that once Theorem~\ref{thm:analytic_solution_extsts_and_unique} is proved, we can apply the classical smooth ODE theory outside the interval $[-\tau;\tau]$.

The proof of Theorem~\ref{thm:analytic_solution_extsts_and_unique} is based on a key estimate, which we shall prove as the separate Lemma~\ref{lm:key_estimate}. Denote
\[
\|\xi\|=\sup_{|t|\le\tau}|\xi(t)|.
\]

\begin{lemma}
	\label{lm:key_estimate}
	Let~$\tau>0$ be an arbitrary number. Suppose that $x,y\in C^2$ for $|t|\le\tau$, and $x(0)=y(0)=\dot x(0)=\dot y(0)=0$. In that case, if $\ddot x(t)\ne 0$ and $\ddot y(t)\ne 0$ for $|t|\le\tau$, then\footnote{The continuous functions~$|\ddot x|$ and~$|\ddot y|$ reach their manima for~$|t|\le\tau$, so the written infima do not vanish.}
	\[
	\left|\frac{\dot x^2}{x} - \frac{\dot y^2}{y}\right| \le 
	\frac{8}{3}\,
	\frac{\|\ddot x\|\|\ddot y\|\|\ddot x-\ddot y\|}
	{\inf_{|t|\le\tau}|\ddot x(t)|\inf_{|t|\le\tau}|\ddot y(t)|}.
	\]
	\noindent This estimate holds for the real case~$x,y,t\in\R$ and for the complex case~$x,y,t\in\mathbb{C}$.
\end{lemma}

\begin{proof}
	Let us estimate the numerator~$\dot x^2 y-x\dot y^2$ from above and the denominator~$xy$ from below. Since~$x=\int_0^t (t-s)\ddot x(s)\,ds$, we have~$|x|\ge \frac12|t|^2\inf|\ddot x|$. The same estimate holds for~$|y|$, so
	\[
	|xy| \ge \frac14|t|^4\inf|\ddot x|\inf|\ddot y|.
	\]
	
	Let us now estimate the numerator from above. Its derivative has the following form:
	\[
	\frac{d}{dt}(\dot x^2 y - x\dot y^2) = \dot x\dot y (\dot x-\dot y) + 2(\dot x y\ddot x - x\dot y \ddot y).
	\]
	
	Since~$\dot x=\int_0^t\ddot x(s)\,ds$, we have~$|\dot x|\le \|\ddot x\||t|$. Similar estimates hold for~$|\dot y|$ and~$|\dot x-\dot y|$. So the first difference in the previous expression has the following estimate:
	\[
	|\dot x\dot y (\dot x-\dot y)| \le 
	\|\ddot x\|\|\ddot y\| \|\ddot x-\ddot y\||t|^3.
	\]
	
	Now we want to estimate the second difference. Let us start with the difference~$\dot xy-x\dot y$. Since
	\[
	\left|\frac{d}{dt}(\dot x y-x\dot y)\right| =
	|\ddot x y-x\ddot y| \le |\ddot x||x-y| + |x||\ddot x-\ddot y| \le
	\|\ddot x\|\|\ddot x-\ddot y\| |t|^2 ,
	\]
	\noindent we obtain
	\[
	|\dot xy-x\dot y| \le 
	\frac13\|\ddot x\|\|\ddot x-\ddot y\||t|^3.
	\]
	
	\noindent and
	\[
	|\dot x y\ddot x - x\dot y \ddot y|\le
	|\dot x||y||\ddot x - \ddot y| + |\ddot y||\dot xy-x\dot y| \le
	\frac56 \|\ddot x\|\|\ddot y\| \|\ddot x-\ddot y\||t|^3.
	\]
	
	Thus,
	\[
	\left|\frac{d}{dt}(\dot x^2 y - x\dot y^2)\right| \le 
	|\dot x\dot y (\dot x-\dot y)| + 2|x\dot y\ddot x - \dot x y \ddot y| \le
	\frac{8}{3} \|\ddot x\|\|\ddot y\| \|\ddot x-\ddot y\||t|^3,
	\]
	\noindent and
	\[
	|\dot x^2 y - x\dot y^2| \le 
	\frac{2}{3} \|\ddot x\|\|\ddot y\| \|\ddot x-\ddot y\||t|^4.
	\]
	\noindent Gathering together the estimates for the numerator and the denominator, we obtain the inequality stated in the lemma.
\end{proof}

\begin{proof}[Proof of Theorem~\ref{thm:analytic_solution_extsts_and_unique}]
	Equation~\eqref{eq:analytic_with_singularity} holds for a.e.~$t$. Consequently, it is equivalent to the following integral representation:
	\begin{equation}
	\label{eq:integral_map_for_analytic_solution}
	x(t) = F(x)(t) \eqdef \dot x(0)t+ \int_0^t (t-s)
	\Big(\lambda \frac{\dot x^2(s)}{x(s)}+g(s,x(s),\dot x(s))\Big)\,ds.
	\end{equation}
	
	\noindent Obviously, we have~$\dot x(0)=0$ for any $C^2$-solution  of the equation~$x=F(x)$, since, in the opposite case, the integrand has a nonintegrable singularity of type~$1/s$ in a neighborhood of~$s=0$.
	
	We start with the existence of a solution. Let us search for a solution in the space~$X$ that consists of functions~$x(t)\in C^2[-\tau;\tau]$ such that~$x(0)=\dot x(0)=0$. Note that~$\|\ddot x\|=\sup_{|t|\le\tau}|\ddot x(t)|$ is a norm on~$X$. Let~$\varepsilon\in(0;1)$. We fix an arbitrary number~$\ddot x_0\ne 0$ and choose $\varepsilon$-ball in~$X$:
	\[
	B_\varepsilon^\tau=\{x\in X:\|\ddot x-\ddot x_0\|\le \varepsilon \ddot x_0\}
	\]
	
	Let us prove that~$F$ is a contraction mapping. Let~$x,y\in B_\varepsilon^\tau$. Then
	\[
	\left|
	\frac{d^2}{dt^2}F(x)(t)-\frac{d^2}{dt^2}F(y)(t)
	\right| \le
	|\lambda|\left|
	\frac{\dot x^2}{x} - \frac{\dot y^2}{y}
	\right| +
	\left|
	g(t,x,\dot x) - g(t,y,\dot y)
	\right|.
	\]
	
	The first term can be estimated by Lemma~\ref{lm:key_estimate}, since $|\ddot x(t)|\ge(1-\varepsilon)|\ddot x_0|\ne0$ and $|\ddot y(t)|\ge(1-\varepsilon)|\ddot x_0|\ne0$. For the second term, we have
	\[
	|g(t,x,\dot x)-g(t,y,\dot y)|\le 
	\|g_x\||x-y| + \|g_y\||\dot x-\dot y|\le
	\big(\frac12\|g_x\||t|^2 + \|g_{\dot x}\||t|\big) \|\ddot x-\ddot y\|.
	\]
	
	Thus, using~$|t|\le\tau$, we can write
	\begin{equation}
	\label{eq:tau_estimate_1}
	\left|
	\frac{d^2}{dt^2}F(x)-\frac{d^2}{dt^2}F(y)
	\right| \le
	\left(
	\frac{8}{3}|\lambda| \left(\frac{1+\varepsilon}{1-\varepsilon}\right)^2  +
	\frac12\|g_x\|\tau^2 + \|g_{\dot x}\|\tau
	\right)\|\ddot x-\ddot y\|.
	\end{equation}
	
	We have~$|\lambda|<3/8$ by assumption. So there exists a number~$0<\rho_0<1$ such that, for any small enough~$\varepsilon>0$ and $\tau>0$, the coefficient of $\|\ddot x-\ddot y\|$ (we denote it by~$\rho(\tau,\varepsilon)$ for short) is less than~$\rho_0$, $\rho(\tau,\varepsilon)\le\rho_0<1$.
	
	Thus, if~$\varepsilon$ and~$\tau$ are small enough, then~$F$ is a contraction mapping on~$B_\varepsilon^\tau$. The mapping~$F$ is contracting regardless of the choice of~$\ddot x_0\ne 0$. However, we still need to prove that it is possible to choose~$\varepsilon$ and $\tau$ so that the image of~$B_\varepsilon^\tau$ under~$F$ is contained in itself (in this case, the completeness of~$X$ guarantees the existence and uniqueness of a solution in~$B_\varepsilon^\tau$). Thus, we need to put
	\[
	\ddot x_0 = \frac{g(0,0,0)}{1-2\lambda}.
	\]
	
	So let us estimate the distance from~$F(x)$ to the center~$\tilde x=\frac12 \ddot x_0 t^2$ of the ball~$B_\varepsilon^\tau$. First, we compute the distance from the center to its image:
	\[
	\left|
	\frac{d^2}{dt^2}F(\tilde x) - \ddot x_0
	\right| \le
	\left|
	\lambda\frac{\dot{\tilde x}^2}{\tilde x} - \ddot x_0 +g(0,0,0)
	\right| +
	\|g_t\| |t| + \frac12\|g_x\||\ddot x_0||t|^2 + \|g_{\dot x}\||\ddot x_0||t|.
	\]
	\noindent The difference in the first modulus vanishes due to the choice of~$\ddot x_0$. Consequently,
	\[
	\left|
	\frac{d^2}{dt^2}F(\tilde x) - \ddot x_0
	\right| \le \tau (a\tau +b),
	\]
	\noindent where we put~$a=\|g_t\| + \|g_{\dot x}\||\ddot x_0|$ and $b=\frac12\|g_x\||\ddot x_0|$ for brevity.
	
	So, for any~$x\in B_\varepsilon^\tau$, we have
	\begin{multline}
	\label{eq:tau_estimate_2}
	\left|
	\frac{d^2}{dt^2}F(x) - \ddot x_0
	\right| \le
	\left|
	\frac{d^2}{dt^2}F(x) - \frac{d^2}{dt^2}F(\tilde x)
	\right| +
	\left|
	\frac{d^2}{dt^2}F(\tilde x) - \ddot x_0
	\right| \le\\
	\le \rho(\tau,\varepsilon) \|x-\ddot x_0\| + \tau(a\tau+b) \le
	\left(
	\rho(\tau,\varepsilon) + \frac{\tau(a\tau+b)}{\varepsilon |\ddot x_0|}
	\right)\varepsilon |\ddot x_0|.
	\end{multline}
	
	Since~$\rho(\tau,\varepsilon)\le\rho_0<1$ for any $\varepsilon$ and $\tau$ small enough, we are able to decrease~$\tau$ in such a way that the right-hand side becomes less than~$\varepsilon |\ddot x_0|$.
	
	So we prove that there exists a~$\tau_0>0$ and an~$\varepsilon_0>0$ such that equation~\eqref{eq:analytic_with_singularity} has a unique solution in~$B_{\varepsilon_0}^{\tau_0}$. We denote this solution by~$\hat x$.
	
	Let us now prove that a solution in the whole class~$C^2[-\tau_0;\tau_0]$ is unique and coincides with~$\hat x$. Let~$x$ be a solution of equation~\eqref{eq:analytic_with_singularity} in this class. Previously we have shown that~$\dot x(0)=0$. Since~$\dot x(t)\ne 0$ in a left (or right) punctured neighborhood of~$t=0$ by the assumptions of the theorem, we can use L'H\^opital's rule for equation~\eqref{eq:analytic_with_singularity}. So
	\[
	\ddot x(0) = 2\lambda \ddot x(0) + g(0,0,0)
	\quad\Rightarrow\quad
	\ddot x(0) = \frac{g(0,0,0)}{1-2\lambda} = \ddot x_0.
	\]
	
	Consequently, there exists a~$\tilde \tau\le \tau_0$ such that the inequality~$\|\ddot x-\ddot x_0\|\le\varepsilon_0|\ddot x_0|$ holds for~$|t|\le\tilde\tau$. Thus, the solutions~$x$ and~$\hat x$ must coincide for the interval~$|t|<\tilde\tau$. Moreover, they coincide outside the interval by the classical Picard theorem on the existence and uniqueness of solutions for ODE with smooth right-hand side.
	
	The last remaining thing is to prove the analyticity of~$\hat x$ under the assumption that the function~$g$ is analytic. Notice that if~$x(t)$ is an analytic function in a neighborhood of~$0$ and~$x(0)=\dot x(0)=0$, but~$\ddot x(0)\ne 0$, then the fraction~$\dot x^2/x$ is analytic. Consequently, the image~$F(x)$ is analytic function too. It is natural to use the Weierstrass theorem on the uniform limit of complex analytic functions. Note that the uniform limit of real-analytic functions on an interval can be a nonanalytic function. So we need to complexify the problem in a standard way.
	
	Let us now assume that~$t\in\mathbb{C}$, $x:\mathbb{C}\to\mathbb{C}$, and the function $g:\mathbb{C}^3\to \mathbb{C}$ denotes the complex-analytic extension of~$g$ on~$\mathbb{C}^3$.
	
	It is obvious that if a complex analytic function~$x(t)$ satisfies equation~\eqref{eq:analytic_with_singularity} for a.e.~$t$ in a neighborhood of~$0$, then it satisfies~\eqref{eq:analytic_with_singularity} for all $t\ne0$ in this neighborhood by continuity. So we are able to search for a solution in the integral form~\eqref{eq:integral_map_for_analytic_solution} again, where this integral should be taken over a segment joining~$0$ and~$t$ in the complex plane.
	
	If~$\ddot x(0)\ne 0$, then the image~$F(x)$ is analytic in the punctured disk~$0<|t|\le\tau$. If~$x(0)=\dot x(0)=0$, then~$F(x)$ is bounded in a neighborhood of~$0$. So, in this case, $F(x)$ is analytic on the whole disk~$|t|\le\tau$ by Riemann's theorem on removable singularity.
	
	All the estimates that we obtain for the real case remain the same if we consider the new space~$X$	of all complex analytic functions in the disk~$|t|\le\tau$ and the corresponding ball~$B_\varepsilon^\tau$ of all functions~$x(t)\in X$ such that~$x(0)=\dot x(0)=0$ and~$\|x-\ddot x_0\|=\sup_{|t|< \tau}|x(t)-\ddot x_0|\le \varepsilon\ddot x_0$. The mapping~$F$ becomes contracting on~$B_\varepsilon^\tau$ under the right choice of~$\varepsilon>0$ and $\tau>0$. Moreover, $F$ maps $B_\varepsilon^\tau$ to itself by the same reason. The space of all complex analytic functions is complete under the uniform norm, so there exists a unique fixed point in the ball~$B_\varepsilon^\tau$. The complex analytic solution found on the disk~$|t|\le \tau$ obviously becomes a real analytic solution if we restrict $t$ to the real line. It remains to say that a $C^2$ real solution is unique and must coincide with the analytic one.
\end{proof}

It is not hard to construct recurrence formulas for the Taylor coefficients of the solution of the Euler--Lagrange equation~\eqref{eq:main_Lagrange_eq}. They are based on the recurrence formula for the inverse of the series~$\sum_{k=0}^\infty a_k t^k$: if~$a_0\ne 0$, then, for the inverse series~$\sum_{k=0}^\infty b_k t^k$, we have
\[
b_0=\frac{1}{a_0}\quad\mbox{and}\quad b_k = -\frac{1}{a_0}\sum_{j=0}^{k-1}a_{k-j}b_j
\mbox{ for }k\ge 1.
\]
\noindent Let us write out a few first derivatives at~$p_0$ of the solution of equation~\eqref{eq:main_Lagrange_eq}:
\[
v(p_0)=p_0;\quad 
v'(p_0)=1; \quad 
v''(p_0)=\frac{3 p_0^2-1}{3 p_0(1+p_0^2)};
\]
\[
v'''(p_0)=\frac{3 p_0^4+2 p_0^2+1}{2 p_0^2(1+p_0^2)^2};\quad
v''''(p_0)=\frac{3 \left(459 p_0^6-108 p_0^4-618 p_0^2-52\right) p_0^2+241}{180 (3 p_0^2-1) p_0^3(1+p_0^2)^3}.
\]

Now we consider the case in which the right-hand side of equation~\eqref{eq:analytic_with_singularity} depends on a parameter $\alpha\in\R$:
\begin{equation}
\label{eq:analytic_with_singularity_and_parameter}
\ddot x = \lambda \frac{\dot x^2}{x} + g(t,x,\dot x,\alpha).
\end{equation}

Denote by~$x(t,\alpha)\in C^2$ the unique $C^2$ solution of this equation with the initial condition~$x(0,\alpha)=0$. Using Theorem~\ref{thm:analytic_solution_extsts_and_unique}, we obtain the following. If~$0<|\lambda|<3/8$, $g(0,0,0,\alpha)\ne 0$, and the functions~$g$, $g_x$, and~$g_{\dot x}$ are continuous, then there exists a unique solution on~$|t|\le\tau_\alpha$ for some positive time~$\tau_\alpha>0$, which depends on~$\alpha$ in general.

\begin{prop}
	\label{prop:uniform_limit}
	Suppose that~$0<|\lambda|<3/8$ and the function~$g$ is continuous together with its partial derivatives~$g_x$, $g_{\dot x}$ and~$g_\alpha$. Let us fix a value~$\alpha_0$ of the parameter~$\alpha$ and a number~$\delta>0$. If~$g(0,0,0,\alpha)\ne 0$ for~$|\alpha-\alpha_0|\le\delta$, then there exist a~$\tau>0$ and a~$C>0$ such that all the solutions~$x(t,\alpha)$ are defined for~$|t|\le\tau$, and~$|\ddot x(t,\alpha)-\ddot x(t,\beta)|\le C|\alpha-\beta|$ for all~$|t|\le\tau$. If, additionally, $g$ is analytic, then the solution~$x(t,\alpha)$ is analytic in~$t$ and~$\alpha$.
\end{prop}

\begin{proof}
	First, we choose a common interval for all the solutions for~$|\alpha-\alpha_0|\le\delta$. In fact, we proved in Theorem~\ref{thm:analytic_solution_extsts_and_unique} that there exists a unique solution~$x(t,\alpha)$ on $|t|\le\tau$ for a given value of~$\alpha$ if there exists an~$\varepsilon>0$ such that the coefficients on the right-hand sides of estimates~\eqref{eq:tau_estimate_1} and~\eqref{eq:tau_estimate_2} are strictly less than~1. So we are able to choose the common~$\tau$ and~$\varepsilon$ for all~$\alpha$. Moreover, in this case, for any~$\alpha$, we have for~$|t|\le\tau$
	\[
	\left|\ddot x(t,\alpha)-\frac{g(0,0,0,\alpha)}{1-2\lambda}\right| \le
	\varepsilon\left|\frac{g(0,0,0,\alpha)}{1-2\lambda}\right|.
	\]
	\noindent Thus, we immediately see that the following estimate holds for all~$\alpha$:
	\[
	\frac{\|\ddot x(\cdot,\alpha)\|}{\inf_{|t|\le\tau}|\ddot x(t,\alpha)|}\le\frac{1+\varepsilon}{1-\varepsilon}.
	\]
	
	Now we define the following functional~$\mathcal{F}$ for~$x(t,\alpha)$:
	\[
	(\mathcal{F}x)(t,\alpha) \eqdef \int_0^t (t-s)
	\Big(\lambda \frac{\dot x^2(s,\alpha)}{x(s,\alpha)}+g(s,x(s,\alpha),\dot x(s,\alpha),\alpha)\Big)\,ds
	\]
	
	If~ $x=\mathcal{F}x$, then, by Lemma~\ref{lm:key_estimate}, we obtain
	\[
	|\ddot x(t,\alpha)-\ddot x(t,\beta)|\le 
	\left(
	\frac{8}{3}|\lambda| \left(\frac{1+\varepsilon}{1-\varepsilon}\right)^2  +
	\frac12\|g_x\|\tau^2 + \|g_{\dot x}\|\tau
	\right)\|\ddot x(\cdot,\alpha)-\ddot x(\cdot,\beta)\| + \|g_\alpha\||\alpha-\beta|.
	\]
	\noindent The coefficient of~$\|\ddot x(\cdot,\alpha)-\ddot x(\cdot,\beta)\|$ on the right-hand side was previously denoted by~$\rho(\tau,\varepsilon)$. The numbers~$\tau>0$ and~$\varepsilon>0$ were chosen in such a way that~$\rho(\tau,\varepsilon)<1$. Consequently,
	\[
	\|\ddot x(\cdot,\alpha)-\ddot x(\cdot,\beta)\| \le \frac{\|g_\alpha\|}{1-\rho(\tau,\varepsilon)}|\alpha-\beta|,
	\]
	\noindent which was needed to be proved.
	
	Now let us suppose additionally that the function~$g$ is analytic. Then, just as in the proof of Theorem~\ref{thm:analytic_solution_extsts_and_unique}, we complexify equation~\eqref{eq:analytic_with_singularity_and_parameter}. We use $\sup_{t,\alpha}|\ddot x(t,\alpha)|$ as a norm on the space of functions~$x(t,\alpha)$ (it is well defined, since~$x(0,\alpha)=\dot x(0,\alpha)=0$). Estimates of the operators~$F$ and~$\mathcal{F}$ coincide, and the space of analytic functions is complete under the uniform norm. So there exists a unique analytic in~$t$ and~$\alpha$ solution of the equation~$x=\mathcal{F}x$.
\end{proof}

So if~$\alpha\to\alpha_0$ and $g(0,0,0,\alpha_0)\ne 0$, then Proposition~\ref{prop:uniform_limit} guarantees that the limit~$x(\cdot,\alpha)\to x(\cdot,\alpha_0)$ is uniform in a neighborhood of~$0$. Note that, outside this neighborhood, the limit is uniform by the classical ODE theory for equations with smooth right-hand side.

\medskip

We also need the following variational equation for equations of type~\eqref{eq:analytic_with_singularity}:
\[
\ddot y = -\lambda \frac{\dot x^2}{x^2}y + 2\lambda \frac{\dot x}{x}\dot y + g_x(t,x,\dot x) y + g_{\dot x}(t,x,\dot x)\dot y.
\]
\noindent Obviously, if~$g\in C^2$, $x\in C^2$, $x(0)=\dot x(0)=0$, $\ddot x(0)\ne 0$ and~$x(t)\ne 0$ for~$t\ne 0$, then the variational equation has the following form:
\begin{equation}
\label{eq:type_of_Jacobi_eq}
\ddot y = 4\lambda\frac{t\dot y-y}{t^2} + \alpha(t)\frac{y}{t} + \beta(t)\dot y + \sigma(t),
\end{equation}
\noindent where~$\alpha,\beta\in C^1$ and~$\sigma\equiv 0$. We add a new function~$\sigma(t)$ to the right-hand side, since it will be very convenient for further study.

\begin{prop}
	\label{prop:var_eq_sol_exists}
	Suppose~$\alpha,\beta,\sigma\in C^1$ and~$\lambda<0$. Then, for any number~$\dot y_0\in\R$, there exists a unique $C^2$-solution~$y$ of the variational equation~\eqref{eq:type_of_Jacobi_eq} with initial conditions~$y(0)=0$ and~$\dot y(0)=\dot y_0$.
	
	Let us now fix~$t_0>0$ and put~$\|y\|=\sup_{|t|\le t_0}|y(t)|$. If~$-\frac12<\lambda<0$ and~$t_0<(1-2|\lambda|)/(\frac12\|\alpha\|+\|\beta\|)$, then  the following estimate holds for the solution~$y$:
	
	\[
	\|\ddot y\| \le 
	\frac{(\|\alpha\| + \|\beta\|)|\dot y_0| + \|\sigma\|}{1-2|\lambda|-(\frac12\|\alpha\|+\|\beta\|)t_0}.
	\]
\end{prop}

\begin{proof}
	Let us use a standard argument. We denote~$z=t\dot y-y$ and obtain the following ODE system:
	
	\[
	\begin{cases}
	\dot t = 1;\\
	\dot y = \frac1t (y+z);\\
	\dot z = 4\lambda \frac{z}{t} + (\alpha + \beta) y + \beta z + t\sigma.
	\end{cases}
	\]
	
	\noindent The solutions of this system are the integral curves of the vector field~$\xi(t,y,z)$ defined by the right-hand side in the space~$\R^3=\{(t,y,z)\}$. This vector field has a singularity for~$t=0$. So we consider a new field~$t\xi(t,y,z)$, which does not have singularities. The integral curves for~$\xi$ and~$t\xi$ coincide, but the speed of movement is different. Denote by~$s$ the time parameter along $t\xi$. So we have
	\[
	\begin{cases}
	t'_s = t;\\
	y'_s = y+z;\\
	z'_s = 4\lambda z + t(\alpha+\beta) y + t\beta z + t^2\sigma.
	\end{cases}
	\]
	\noindent Here the right-hand side is smooth. The origin~$(0,0,0)$ is a fixed point. The linearization at the origin has the following form:
	
	\[
	\begin{pmatrix}
	t'_s\\y'_s\\z'_s
	\end{pmatrix}
	=
	\begin{pmatrix}
	1 & 0 & 0\\
	0 & 1 & 1\\
	0 & 0 & 4\lambda
	\end{pmatrix}
	\begin{pmatrix}
	t\\y\\z
	\end{pmatrix}.
	\]
	\noindent The given matrix is diagonalizable and has eigenvalues~$4\lambda$ and~$1$ (of algebraic multiplicity~$2$). The corresponding eigenvectors are~$(0,a,(4\lambda-1) a)$ for~$4\lambda$ and~$(a,b,0)$ for~$1$. Since~$\lambda<0$, the system is hyperbolic at~$0$. It has a 2-dimensional unstable manifold~$\mathcal{M}$ (which is tangent to the plane~$\{z=0\}$ at~$0$) and 1-dimensional stable manifold (which is tangent to the line~$\{t=0,z=(4\lambda-1) y\}$ at~$0$). We are interested in the invariant surface~$\mathcal{M}$. Let us restrict the vector field~$t\xi$ to~$\mathcal{M}$. The origin is a fixed point on~$\mathcal{M}$, and the linearization yields the unit matrix. It is well known (see, for example, \cite[Lemma~5 on p.~97]{ZelikinOPU}) that, in this case, for any tangent vector~$\eta$ to~$\mathcal{M}$ at~$0$, there exists a unique integral curve of~$t\xi$ that is tangent to~$\eta$ at~$0$. So we take~$\eta=\pm(1,\dot y_0,0)$, and this gives us the existence and uniqueness for the solutions of the variational equation for small enough~$t$. For the other values of~$t$, we are able to use the classical ODE theory for linear equations.
	
	Let us now prove the estimate stated in the proposition. Since~$z=t\dot y-y$ vanish at~$t=0$, we have
	\[
	|t\dot y(t)-y(t)| = \Big|\int_0^t s\ddot y(s)\,ds\Big| \le \frac12 |t|^2\|\ddot y\|.
	\]
	\noindent For~$y$ and~$\dot y$, it follows that $y(t)=\dot y_0t+\int_0^t(t-s)\ddot y(s)\,ds$ and~$\dot y(t)=\dot y_0+\int_0^t\ddot y(s)\,ds$. So, using equation~\eqref{eq:type_of_Jacobi_eq}, we obtain
	\[
	\|\ddot y\| \le \big(2|\lambda| + \frac12\|\alpha\|t_0+\|\beta\|t_0\big)\|\ddot y\|+
	(\|\alpha\|+\|\beta\|)|\dot y_0| + \|\sigma\|,
	\]
	\noindent which immediately leads to the required estimate.
\end{proof}

\begin{corollary}
	\label{cor:est_Jacobi_difference}
	Let~$y$ and~$\tilde y$ be solutions of equation~\eqref{eq:type_of_Jacobi_eq} with $y(0)=\tilde y(0)=0$ and $\dot y(0)=\dot{\tilde y}(0)=\dot y_0$ having the same value of the parameter~$-\frac12<\lambda<0$, but different values of the parameters~$(\alpha,\beta,\sigma)$ and~$(\tilde\alpha,\tilde\beta,\tilde\sigma)$. Let~$t_0<\min\Big\{\frac{1-2|\lambda|}{\frac12\|\alpha\|+\|\beta\|},\frac{1-2|\lambda|}{\frac12\|\tilde\alpha\|+\|\tilde\beta\|}\Big\}$. Then\footnote{We can replace~$\|\dot{\tilde y}\|$ by~$\|\dot{\tilde y}\|\le t_0\|\ddot{\tilde y}\|$ into the inequality below, where the last term can be estimated by Proposition~\ref{prop:var_eq_sol_exists}.}
	\[
	\|\ddot y-\ddot{\tilde y}\|\le
	\frac{ (\|\alpha-\tilde\alpha\| + \|\beta-\tilde\beta\|)\|\dot{\tilde y}\| + \|\sigma-\tilde\sigma\|}
	{1-2|\lambda|-(\frac12\|\alpha\|+\|\beta\|)}.
	\]
\end{corollary}

\begin{proof}
	Put~$w=y-\tilde y$. Then
	\[
	\ddot w = 4\lambda\frac{t\dot w-w}{t^2} + \alpha(t)\frac{w}{t} + \beta(t)\dot w + 
	(\alpha-\tilde\alpha)\frac{\tilde y}{t} + (\beta-\tilde\beta)\dot{\tilde y} + (\sigma-\tilde\sigma).
	\]
	\noindent Obviously $w\in C^2$, $w(0)=\dot w(0)=0$, $\frac{\tilde y}{t}\in C^1$ and~$\dot{\tilde y}\in C^1$. So. using Proposition~\ref{prop:var_eq_sol_exists}, we can write
	\[
	\|\ddot w\|\le
	\frac{\|(\alpha-\tilde\alpha)\frac{\tilde y}{t} + (\beta-\tilde\beta)\dot{\tilde y} + (\sigma-\tilde\sigma)\|}{1-2|\lambda|-(\frac12\|\alpha\|+\|\beta\|)}.
	\]
	\noindent Since $|\tilde y|\le|t|\|\dot{\tilde y}\|$, we have proved the stated estimate.
\end{proof}

\section{Construction of the field of extremals}
\label{sec:field_of_extremals}

\begin{figure}[ht]
	\centering
	\includegraphics[width=0.4\textwidth]{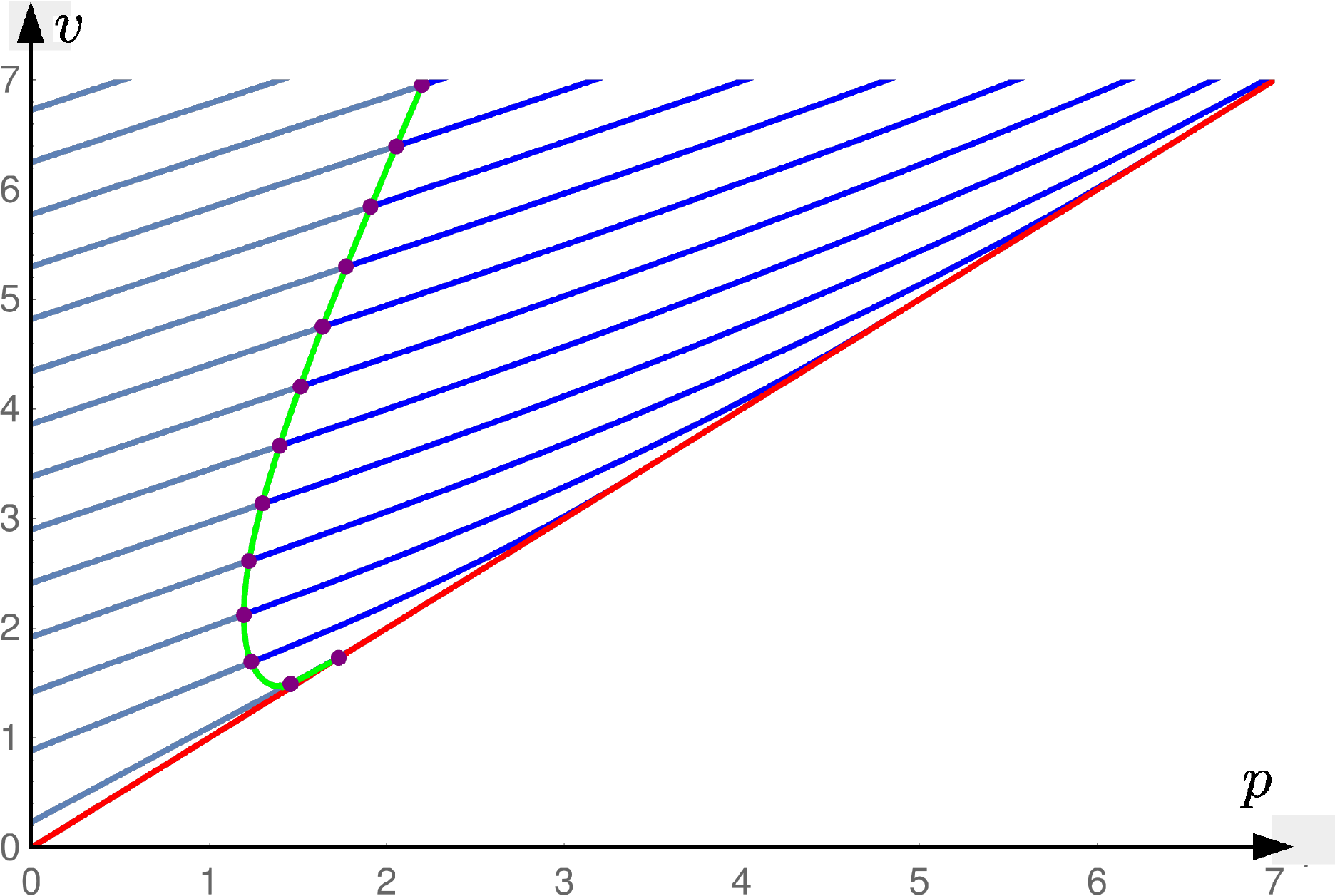}
	\caption{Field of extremals for Pontryagin's Maximum Principle~\eqref{eq:PMP} in the domain~$U=\{0\le p\le v\}$.}
	\label{fig:opt_syn_waxwell}
\end{figure}

In the present section, we shall construct a family of extremals in the domain~$U=\{0\le p\le v\}$. Namely, for any sufficiently large~$p_0$, we shall prove that there exists a $PC^2$-trajectory~$v(p;p_0)$ for $p\in[0;p_0]$ that satisfies all the conditions of Pontryagin's Maximum Principle~\eqref{eq:PMP} with $\lambda_0=1$. It has the following structure: there exists a point~$r(p_0)\in(0;p_0)$ such that the trajectory~$v$ is nonsingular on $[0;r(p_0)]$ (i.e., $v'' = 0$ for~$p\in[0;r(p_0)]$) and is singular on~$[r(p_0);p_0]$ (i.e., condition~\eqref{eq:main_Lagrange_eq} is fulfilled for~$p\in[r(p_0),p_0]$). If~$P_0$ is large enough, then the extremals~$v(p,p_0)$ do not intersect one another for different~$p_0\ge P_0$. Moreover, they one-to-one fill a subdomain in~$U=\{0\le p\le v\}$, which lies above the extremal~$v(p,P_0)$.

We shall prove that Jacobi's equation along any extremal constructed in this section has a unique solution with initial conditions~$\xi(p_0)=0$, $\xi'(p_0)=1$. Moreover, this solution has no conjugate points on half-open segment\footnote{This means that $\xi(p)\ne 0$ for $p\in[0;p_0)$.}~$[0;p_0)$.

The extremals constructed in this section have the following asymptotics:

\begin{enumerate}
	\item \label{list:asymptotics} The switch time moment~$r(p_0)$ is an analytic function on~$p_0$ and has the form~$r(p_0)=p_0(\hat r + O(p_0^{-2}))$ as~$p_0\to+\infty$, where
	\[\hat r=0.108984\pm10^{-6}.\]
	
	\item The height~$M(p_0)=v(0,p_0)$ is an analytic function on~$p_0$ and has the form~$M(p_0)=p_0(\hat M + O(p_0^{-2}))$ as~$p_0\to+\infty$, where
	\[\hat M=0.315736 \pm 10^{-6}.\]
	
	\item The derivative\footnote{This derivative determines the length of the segment that is the intersection of the optimal convex body with the bottom limiting hyperplane~$u=-M$.}~$v'(0,p_0)$ is an analytic function on~$p_0$ and has the form $v'(0,p_0) = v'_0 + O(p_0^{-2})$ as~$p_0\to+\infty$, where
	\[v'_0=0.530068\pm 10^{-6}.\]
	
	\item The value of the functional~\eqref{problem:main_maxwell} is an analytic function on~$p_0$ and has the form~$J(v(p,p_0))= p_0^{-2}(\hat J + O(p_0^{-2}))$ as~$p_0\to+\infty$, where
	\[\hat J=10.7344\pm 10^{-4}.\]
\end{enumerate}

\begin{figure}[ht]
	\centering
	\includegraphics[width=0.4\textwidth]{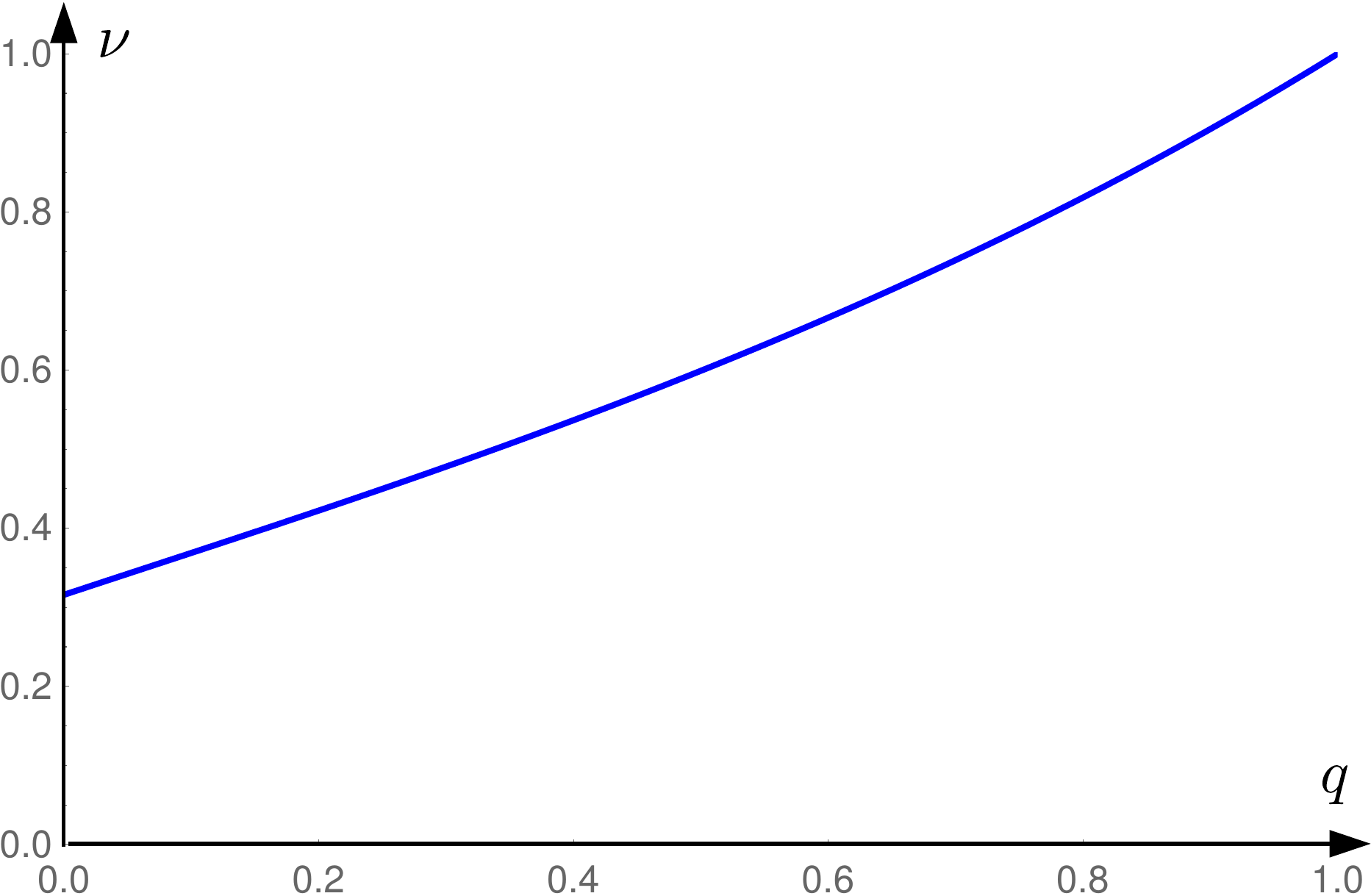}
	\ \ \ \ \ 
	\includegraphics[width=0.4\textwidth]{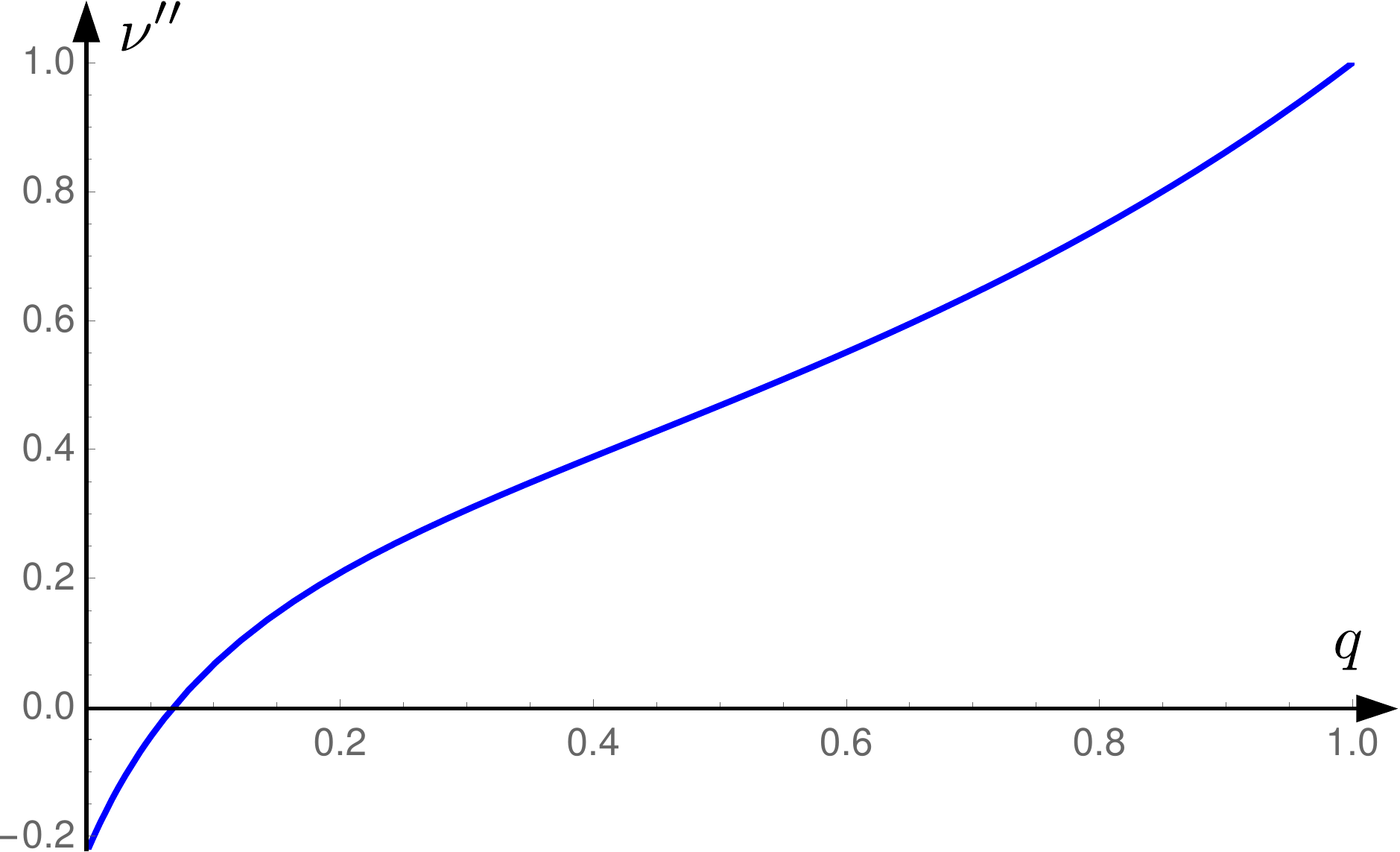}
	\caption{The limit solution of equation~\eqref{eq:nu} with $\alpha=0$ and $\nu(1)=\nu'(1)=1$ (on the left), and its second derivative (on the right).}
	\label{fig:nu_limit}
\end{figure}

We start with constructing of the described extremals. We need to consider the limit situation as~$p_0\to+\infty$. The construction of any trajectory starts from the right end. It should satisfy equation~\eqref{eq:main_Lagrange_eq} for~$p\in[r(p_0);p_0]$. Let us make the change of variables~$\nu(q) = \frac{1}{p_0}v(p_0q,p_0)$. Then equation~\eqref{eq:main_Lagrange_eq} takes the following form:

\begin{equation}
\label{eq:nu}
\nu''= -\frac14\frac{(\nu'-1)^2}{\nu-q} - \frac14\frac{(\nu'+1)^2}{\nu+q} + \frac{2\nu\nu'^2}{\nu^2+\alpha},
\end{equation}
\noindent where~$\alpha=1/p_0^2$ is a small parameter. The initial conditions~$v(p_0,p_0)=p_0$ and $v'(p_0,p_0)=1$ become~$\nu(1)=\nu'(1)=1$. Thus, for any~$\alpha\ne-1$, there exists a unique solution~$\nu(q,\alpha)$ by Theorem~\ref{thm:analytic_solution_extsts_and_unique} with $\nu'\ne 1$ in a punctured neighborhood of~$q=1$, and this solution is analytic. Sometimes we shall omit the dependence on~$\alpha$ and write~$\nu(q)$. In this case, the derivatives~$\nu'(q)$ and $\nu''(q)$ are taken wrt~$q$. The limit case arises when~$\alpha=0$ (the limit solution is depicted in Fig.~\ref{fig:nu_limit}).

Consider the limit case~$\alpha=0$. In this case, equation~\eqref{eq:nu} has an obvious group of symmetries: the uniform stretching in $q$ and $\nu$ does not change the equation\footnote{Equation~\eqref{eq:nu} can be considered as a Hamilton equation  with 1.5 degrees of freedom, since it comes from an Euler--Lagrange equation. But, unfortunately, this symmetry group does not preserve the corresponding symplectic form.}. Consequently, the order of this equation can be lowered (see~\cite{Ovsiannikov,IbragimovEng,IbragimovAzbuka,IbragimovOpyt}). This group has obvious invariants~$t=\nu/q$ and~$x=(q\nu'-\nu)/q$. Therefore, the fraction~$dx/dt=\nu''q/x-1$ is also an invariant. So we obtain the following equation by substituting these invariants into~\eqref{eq:nu}:
\begin{equation}
\label{eq:x_t_invariants}
\frac{dx}{dt} = 2+\frac32\left(\frac{x}{t} +\frac{t}{x}\right)-\frac{x}{2 t (t^2-1)}.
\end{equation}

This equation is called an Abel equation of the second kind. Certain equations of this type that can be integrated explicitly are known. To check this, the equation should be represented in the standard form as $y\,dy/ds-y=g(s)$ with a function $g$. Put $y=xt^{-2}\sqrt[4]{1-t^2}$. Then
\begin{equation}
\label{eq:y_t_invariants}
y\frac{dy}{dt} = 2\frac{\sqrt[4]{1-t^2}}{t^2}y+\frac{3\sqrt{1-t^2}}{2t^3}.
\end{equation}
\noindent The change of time 
\[
s=2\int \frac{\sqrt[4]{1-t^2}}{t^2}\,dt = 
- \,_2F_1\left(\frac{1}{2},\frac{3}{4};\frac{3}{2};t^2\right)t-
2\frac{\sqrt[4]{1-t^2}}{t},
\]
\noindent (where $_2F_1$ is the hypergeometric function) gives $t=t(s)$ as an implicit function and equation~\eqref{eq:x_t_invariants} takes the standard form
\begin{equation}
\label{eq:y_s_invariants}
y\frac{dy}{ds} -y = \frac{3\sqrt[4]{1-t^2}}{4t}.
\end{equation}

We are interested in the particular solution of this equation given by the initial data $\nu(1)=\nu'(1)=1$. Suppose that it is possible to find this solution explicitly in the form~$\Phi(s,y)=0$, which is equivalent to $\Psi(t,x)=0$ with the appropriate function~$\Psi$. We claim that, in this case, the first-order ODE~$\Psi(\frac{\nu}{q},\frac{q\nu'-\nu}{q})=0$ can be solved by quadratures. Indeed, it also respects the mentioned group of symmetries, and, using the coordinates $\tau=\ln q$ and $t=\frac vq$, we get an autonomous equation~$\Phi(t,\frac{dt}{d\tau})=0$ on~$t(\tau)$, which can be solved by quadratures. But, unfortunately, we were unable to explicitly solve equation~\eqref{eq:nu} even for the case~$\alpha=0$. It seems that equations~\eqref{eq:x_t_invariants}, \eqref{eq:y_t_invariants}, and~\eqref{eq:y_s_invariants} representing the case~$\alpha=0$ cannot be solved by quadratures (Julia's method~\cite{Julia} is inapplicable here, and the most famous book with lists of explicitly integrable equations~\cite{PolyaninZaitsev} does not contain these equations). So we are forced to use the numeric solution of equation~\eqref{eq:nu} for~$\alpha=0$. Let us remark that equation~\eqref{eq:nu} has a singularity at~$q=1$ (regardless of~$\alpha$). So all numeric computations need additional accuracy. But it is easy to work with this singularity. Indeed, the solution is analytic by Theorem~\ref{thm:analytic_solution_extsts_and_unique}, and it is easy to find its derivatives at~$q=1$:
\[
\nu(1)=1,\quad \nu'(1) = 1,\quad 
\nu''(1) = \frac{3-\alpha}{3(1+\alpha)},\quad 
\nu'''(1) = \frac{3+2\alpha+\alpha^2}{2(1+\alpha)^2},\quad\ldots
\]
\noindent Therefore, it is sufficient to move away from this singularity using a Taylor polynomial. We emphasize again that we need to use numerical methods for solving equation~\eqref{eq:nu} only at~$\alpha=0$ (but it is easy to solve it numerically for other values of~$\alpha$, which we do not need). So we denote the limit solution by ~$\hat\nu(q)=\nu(q,0)$.

We claim that the switch time~$r(p_0)$ is determined by conditions~\eqref{eq:psi_p_not_sing} for~$p_1=0$ and $p_2=r(p_0)$. Indeed, the condition~$\psi(r(p_0))=0$ is fulfilled automatically (since we are moving away from the singular part of the trajectory), and~$\psi(0)=0$ by the orthogonality conditions of Pontryagin's Maximum Principle. Since~$v''=0$ for $p\in(0;r(p_0))$, we have
\[
\ddt f_{v'} - f_v = f_{pv'} + f_{vv'}v' -f_v.
\]
\noindent Since $\psi(r(p_0))=\psi'(r(p_0))=0$, the condition~$\psi(0)=0$ has the following form:
\begin{equation}
\label{eq:I_def}
\psi(0)=\int_0^{r(p_0)} p (f_{pv'} + f_{vv'}v' - f_v)dp=0.
\end{equation}
\noindent Here we must substitute into the integral an affine function~$v$ on~$p$ of the form~$v'(r(p_0))(p-p_0) + v(r(p_0))$. Let us rewrite this equation in new coordinates. Denote~$\rho = r(p_0)/p_0$ and~$\eta(q)=\nu'(q)(q-\rho)+\nu(\rho)$. Then the equation~$\psi(0)=0$ has the following form (up to the multiplier~$-p_0^{-2}$):
\[
-p_0^{-2}\psi(0)=I(\rho,\alpha)\eqdef\int_0^\rho q(g_\eta - g_{q\eta'} - g_{\eta\eta'}\eta')\,dq=0,
\]
\noindent where
\[
g(q,\eta,\eta') = p_0^3 f(p_0q,p_0\eta,\eta') =
\frac{2\sqrt{\eta^2-q^2}}{(\eta^2+\alpha)^2}\eta'^2 -
\frac{q\eta'-\eta}{\eta(\eta^2+\alpha)\sqrt{\eta^2-q^2}}.
\]

Our goal is to prove the existence of a solution $\rho(\alpha)\in(0;1)$ for the equation $I(\rho,\alpha)=0$ at a given value of $\alpha$.

\begin{figure}[ht]
	\centering
	\includegraphics[width=0.4\textwidth]{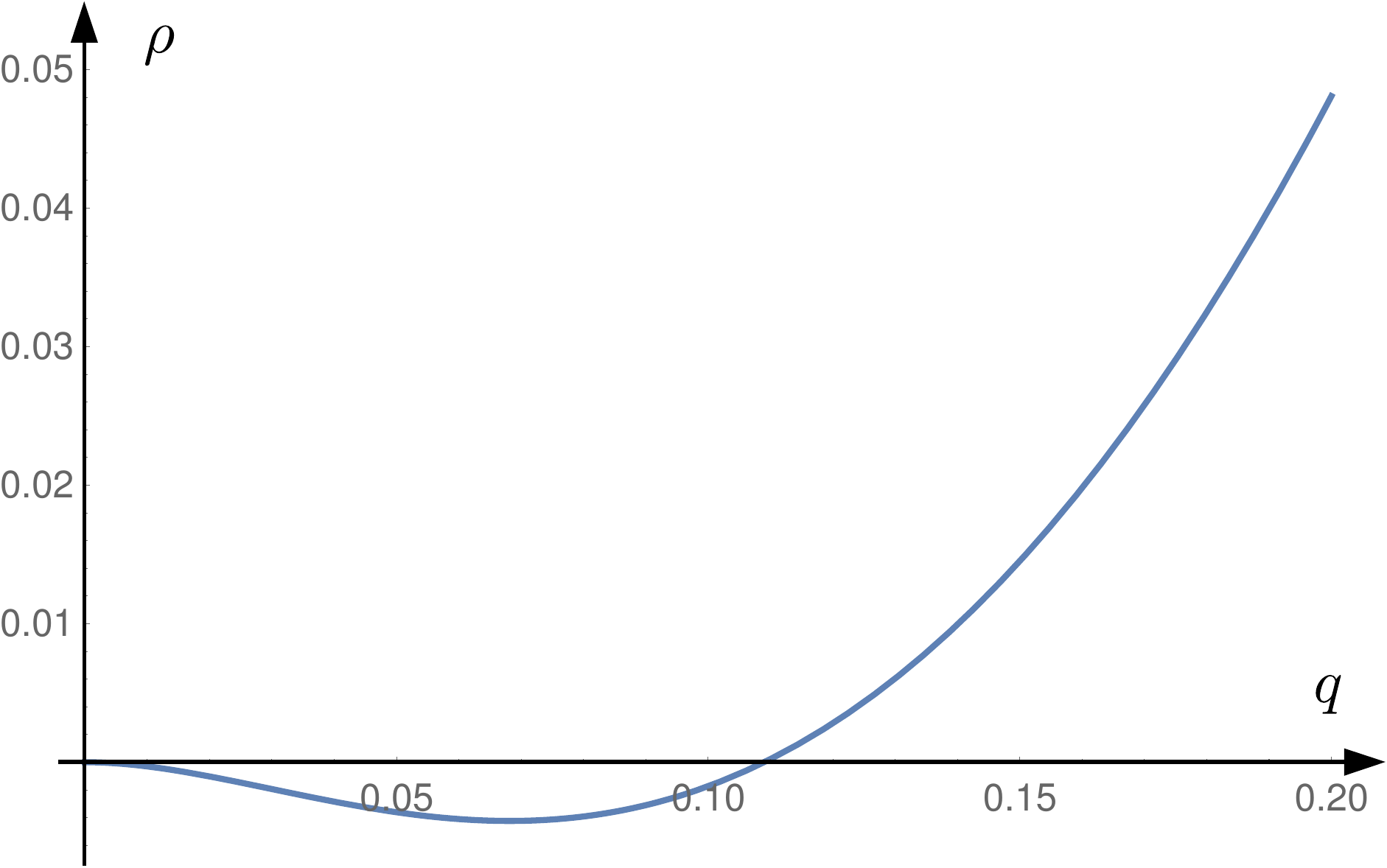}
	\caption{Graph of the function~$I(\rho,0)$.}
	\label{fig:I_null_rho}
\end{figure}

We start by studying the limit case~$\alpha=0$. For~$\alpha=0$, the integral~$I(\rho,0)$ can be explicitly expressed by the solution~$\hat\nu(\rho)$ of equation~\eqref{eq:nu} (see Fig.~\ref{fig:I_null_rho}):
\begin{multline*}
I(\rho,0) = \frac{1}{4(\rho\hat\nu'-\hat\nu)^2}
\Big[
3\hat\nu'\arcsin\frac{\rho}{\hat\nu} -
2-2\hat\nu'^2+\frac{1}{\hat\nu^4}\sqrt{\hat\nu^2-\rho^2}
\big((\hat\nu-\rho\hat\nu')^3 + \\
+(\hat\nu-\rho\hat\nu')^2\rho^2\hat\nu'^2 + \hat\nu^3(1+2\hat\nu'^2)\big)
\Big].
\end{multline*}
\noindent Therefore, the function~$I(0,\rho)$ tends to $+\infty$ as~$\rho\to1-0$. Put~$\rho=1-\Delta\rho$. Then
\[
I(1-\Delta\rho,0) = \Delta\rho^{-2}
\left(
-1 + \frac{3\pi}{8} + O(\Delta\rho)
\right).
\]

It is easy to see that $I(0,0)=I'_{\rho}(0,0)=0$ and~$I''_{\rho\rho}(0,0) = \frac{2(3\hat\nu'(0)^2-1)}{\hat\nu(0,0)^4}<0$, because
\[
\hat\nu(0)=0.3157595\pm 10^{-6}\ne 0
\quad\mbox{and}\quad
\hat\nu'(0)=0.5350553\pm10^{-6}<\frac{1}{\sqrt{3}}. 
\]

Thus, there exists a solution of the equation~$I(\rho,0)=0$. We are able to construct the function~$\hat\nu(\rho)$ only numerically, so the constructed solution of the equation~$I(\rho,0)=0$ can only be found  numerically:
\[
\hat r=\rho(0) = 0.108984\pm10^{-6}
\qquad
I'_{\rho}(0,\hat r) = 0.220371\pm10^{-6}\ne 0.
\]

Note that $\hat\nu''(q)>0$ for $q\ge 0.1$ (see Fig.~\ref{fig:nu_limit}). Consequently, the limit solution~$\hat\nu$ is strictly convex for $q\in[\hat r,1]$.	

Let us now investigate the behavior of~$I$ as a function of~$\rho$ for small values of the parameter~$\alpha>0$. The main tool here is Proposition~\ref{prop:uniform_limit}, which states that~$\nu(q,\alpha)\to\hat\nu(q)$ for all~$q\in[0;1]$ uniformly in~$\alpha\to+0$. Moreover, the complex-analytic extensions also uniformly converge. Thus, all the derivatives~$\nu^{(k)}(q,\alpha)$ also uniformly converge. It remains to note that~$I'_\rho(\hat r,0)\ne 0$. Therefore, the equation~$I(\rho,\alpha)=0$ has a solution~$\rho(\alpha)$ for all~$\alpha$ that are small enough (i.e., for all~$p_0$ large enough), and this solution has asymptotics~$\rho(\alpha) = \hat r + O(\alpha)$.

\begin{figure}[ht]
	\centering
	\includegraphics[width=0.4\textwidth]{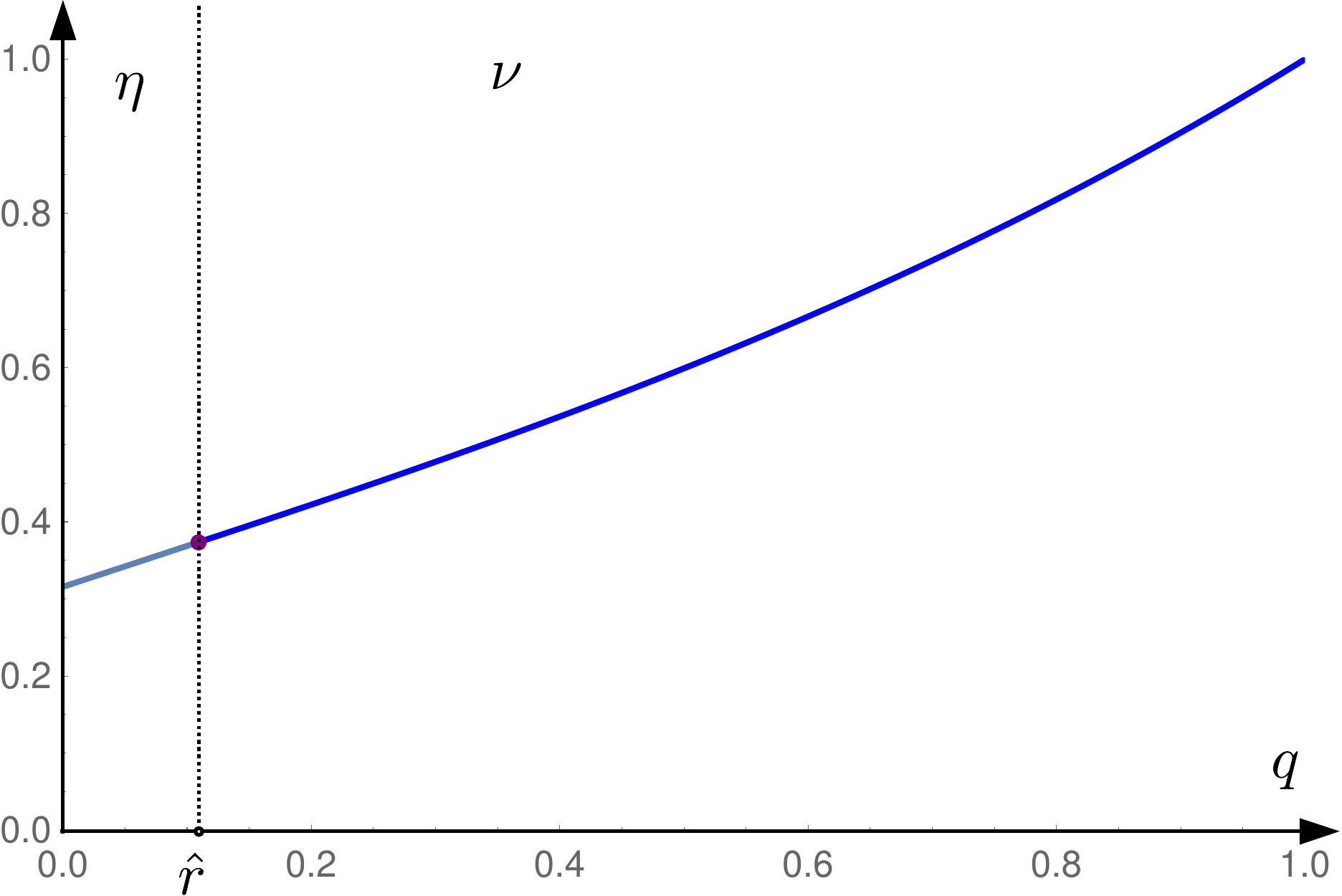}
	\caption{The limit extremal at $\alpha=0$.}
	\label{fig:nu_limit_with_line}
\end{figure}

So the limit trajectory (at~$\alpha=0$) has the following structure (see Fig.~\ref{fig:nu_limit_with_line}): (i) it coincides with the solution of equation~\eqref{eq:nu} for~$q\ge\hat r$, (ii) it is an affine function for~$q\le\hat\rho$, and (iii) it is continuous together with its first derivative at the switching point~$q=\hat\rho$.

Thereby, in the original coordinates~$(p,v)$, we see that, for any large enough~$p_0$, there exists a trajectory of the type described at the beginning of the section. This trajectory is strictly convex for~$p\in[r(p_0);p_0]$, where the switching point~$r(p_0)$ satisfies equation~\eqref{eq:psi_eq_int}. Some numerically computed trajectories for~$p_0>\sqrt3$ are depicted in Fig.~\ref{fig:opt_syn_waxwell} (see also Hypothesis~\ref{hyp1} below).

Let us now check the conditions of Pontryagin's Maximum Principle. First, we check that the control~$\theta=v''$ satisfies the maximum condition $\psi\theta\to\max_{\theta\ge 0}$. We know that the adjoint variable~$\psi$ is identically~$0$ on the singular part (and so any nonnegative control is allowed). Outside the singular part, we have $\theta\equiv0$. So we should check that $\psi\le 0$ on $[0;r(p_0)]$. The adjoint variable~$\psi$ can be found from the equation~$\psi''=\frac{d}{dp}f_{v'}-f_v$. Since~$v''\equiv0$ on $[0;r(p_0)]$ and~$\psi(r(p_0))=0$, we have
\[
\psi(\tilde p) = 
\int_{\tilde p}^{r(p_0)} (\tilde p-p)(f_v-f_{pv'} - f_{vv'}v')\,dp.
\]

\begin{figure}[ht]
	\centering
	\includegraphics[width=0.4\textwidth]{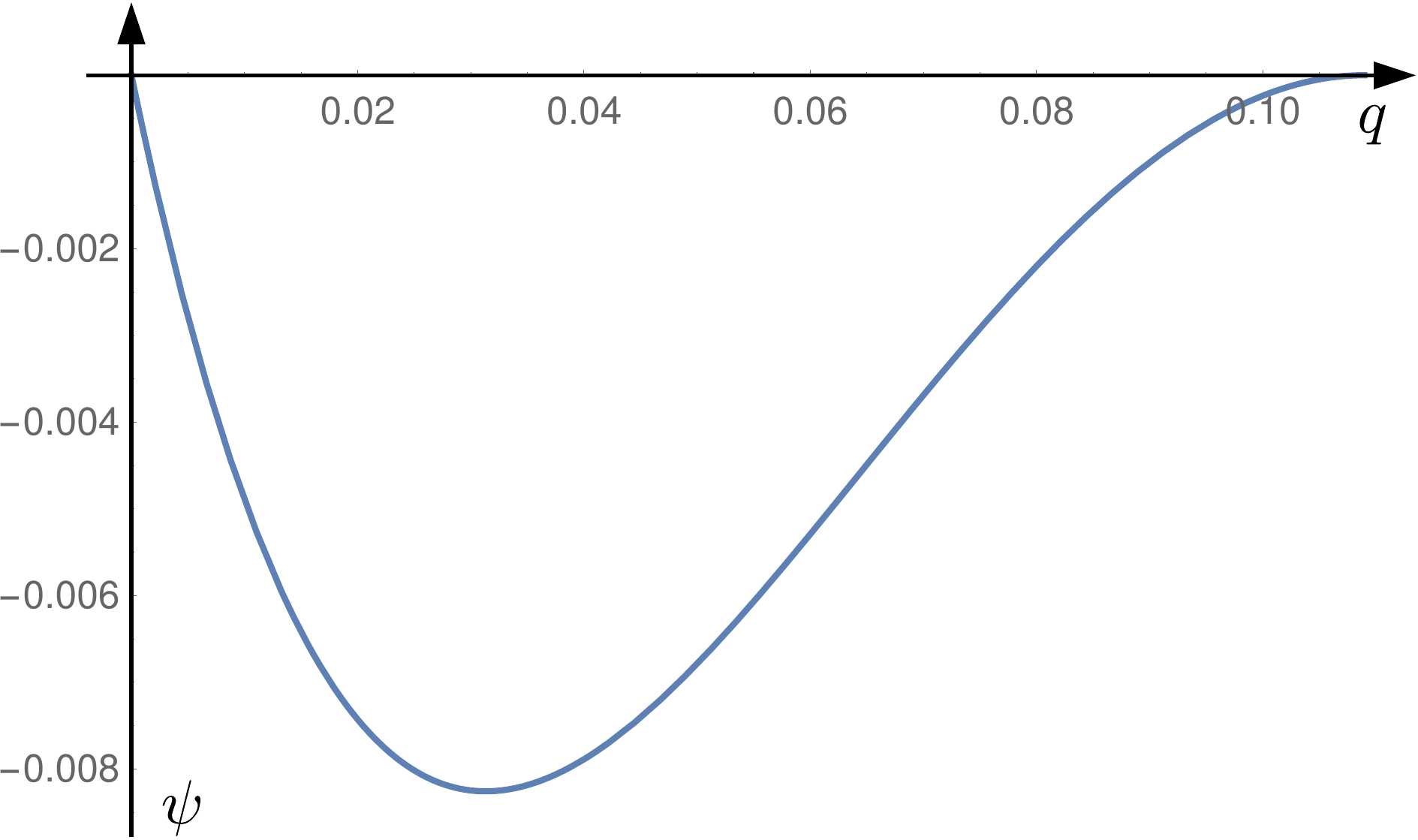}
	\caption{Graph of the function~$\omega(q)$ for $q\in[0;\hat r]$ at $\alpha=0$.}
	\label{fig:omega_limit}
\end{figure}

Now we pass to the limit~$p_0\to+\infty$. Put $p=p_0 q$ and $v=p_0\nu$ as usual. Then, for~$\tilde p=p_0\tilde q$ and~$\omega(q)=p_0^{-2}\psi(p_0q)$, we have
\[
\omega(\tilde q)=p_0^{-2}\psi(p_0\tilde q) = 
\int_{\tilde q}^{\rho(\alpha)} (\tilde q-q)(g_\nu-g_{q\nu'} - g_{\nu\nu'}\nu')\,dp.
\]
\noindent The limit function~$\omega(q)$ at $\alpha=0$ is depicted in Fig.~\ref{fig:omega_limit}. Obviously, $\omega(\rho(\alpha))=\omega'(\rho(\alpha))=0$ and $\omega''(\rho(\alpha))=-g_{\nu'\nu'}\nu''(\rho(\alpha))<0$. Also~$\omega(0)=0$ for small enough~$\alpha$, since the switching time~$\rho(\alpha)$ satisfies the condition~$I(\rho,\alpha)=0$. The function~$\omega(q)$ numerically found at~$\alpha=0$ is negative on~$(0;\hat r)$ and~$\omega'(0)<0$. Thus, the functions~$\omega(q)$ are negative on~$(0;\rho(\alpha))$ for small enough~$\alpha$ by Proposition~\ref{prop:uniform_limit}. Therefore, $\psi(0)=0$ and $\psi(p)<0$ for $p\in(0;r(p_0))$ for large enough~$p_0$. 

Second, we check the orthogonality conditions at the right end. Since the trajectory is singular on~$[r(p_0),p_0]$, we have $\psi\equiv0$, $\varphi\equiv f_{v'}$ and~$H=f_{v'}v'-f$. Thus, the condition $\psi(p_0)=0$ is fulfilled automatically. We know that~$\varphi\to\infty$ and~$H\to\infty$ as~$p\to p_0-0$. Let us check that~$\varphi-H\to0$ as~$p\to p_0-0$. We have
\[
H-\varphi = \frac{2\sqrt{v^2-p^2}}{(1+v^2)^2}v'(v'-2) -
\frac{\sqrt{v-p}}{v(1+v^2)\sqrt{v+p}}=O(p-p_0)\to0.	
\]
\noindent Third, since $\psi(0)=0$ by~$I(\rho,\alpha)=0$, the orthogonality condition at the left end is fulfilled. Consequently, the constructed trajectories satisfy Pontryagin's Maximum Principle and are extremals in problem~\eqref{problem:main_maxwell}.

Now we show that if~$P_0$ is large enough, then the extremals~$v(p,p_0)$ for~$p_0\ge P_0$ do not intersect each other. Let us denote by~$\kappa(q,\alpha)$ the trajectory made up of~$\eta(q,\alpha)$ and~$\nu(q,\alpha)$, i.e., $\kappa=\eta$ for $q\in[0;\rho(\alpha)]$ and $\kappa=\nu$ for $q\in[\rho(\alpha);1]$. Consider new coordinates $(q,\alpha)$ in the plane~$\R^2=\{(p,v)\}$: $p=q/\sqrt{\alpha}$ and $v=\kappa(q,\alpha)/\sqrt{\alpha}$. We claim that the Jacobian of this change of variables does not vanish (this will immediately prove that the trajectories $v(p,p_0)=p_0\kappa(p/p_0,p_0^{-2})$ do not intersect for different values of~$p_0$). We have
\[
\Delta=\det\begin{pmatrix}
p'_q & p'_\alpha \\
v'_q & v'_\alpha
\end{pmatrix}
=
\det\begin{pmatrix}
\alpha^{-\frac12}         & -\frac12\alpha^{-\frac32}q \\
\alpha^{-\frac12} \kappa' & -\frac12\alpha^{-\frac32}\kappa + \alpha^{-\frac12}\frac{\partial\kappa}{\partial\alpha}
\end{pmatrix}
=
\frac12\alpha^{-2}\left(
q\kappa' - \kappa + 2\alpha \frac{\partial\kappa}{\partial\alpha}
\right).
\]
\noindent The expression on the right-hand side in brackets is an analytic function of~$q$ and~$\alpha$ for~$q>\rho(\alpha)$. Since $\nu(1,\alpha)\equiv1$, $\nu'(1,\alpha)\equiv 1$ and $\nu''(1,\alpha)=(3-\alpha)/(3+3\alpha)$, the Taylor decomposition at $q=1$, $\alpha=0$ gives
\[
\Delta = \frac12\alpha^{-2} \Big( q-1 + O\big(\alpha^2+(1-q)^2\big)\Big).
\]
\noindent Thus, there exists a left neighborhood $q\in(1-\delta q,1)$, $\delta q>0$, where~$\Delta$ is negative for all small enough~$\alpha>0$. Since the function~$q\kappa'-\kappa$ is monotonic, $\Delta$ is also negative on the segment $q\in[0;1-\delta q]$ for all small enough~$\alpha>0$, q.e.d.

\begin{figure}[ht]
	\centering
	\includegraphics[width=0.4\textwidth]{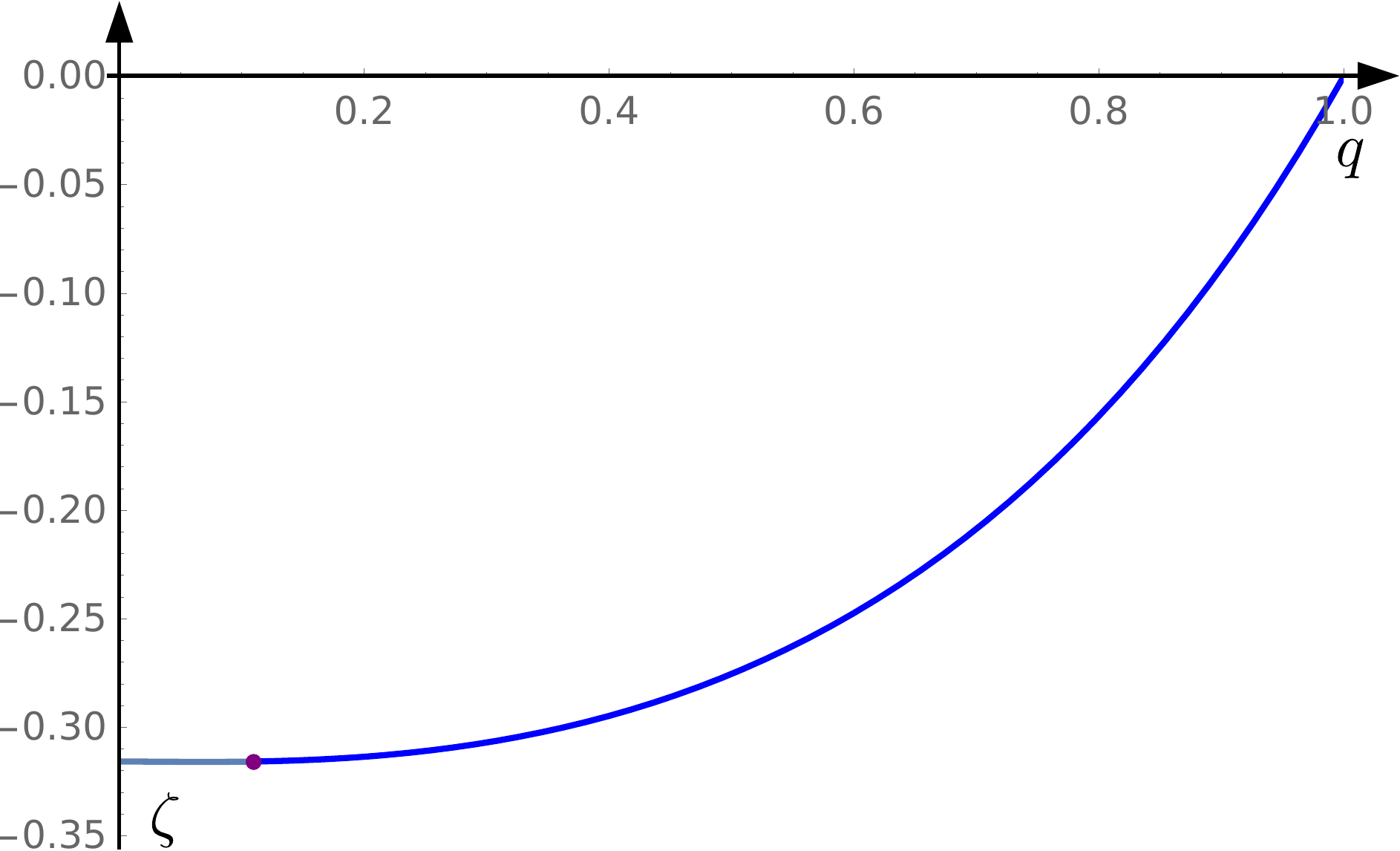}
	\caption{The solution~$\zeta$ of the variational equation~\eqref{eq:nu} at~$\alpha=0$.}
	\label{fig:Jacobi_limit}
\end{figure}

Let us now prove that, for all large enough~$p_0$, there exists a unique solution~$\xi(p,p_0)$ of the Jacobi equation~$\frac{d}{dp}(f_{v'v'}\xi'+f_{v'v}\xi)=f_{v'v}\xi'+f_{vv}\xi$ with initial conditions~$\xi(p_0,p_0)=0$ and~$\xi'(p_0,p_0)=1$, and this solution does not have conjugate points on~$[0;p_0)$. We emphasize the following important difference from the classic case: we substitute the constructed extremal~$v(p,p_0)$ into Jacobi equation, but this extremal satisfies the Euler--Lagrange equation~$\frac{d}{dp}f_{v'}-f_v=0$ only for~$p\in[r(p_0),p_0]$, but, for~$p\in[0;r(p_0)]$, it does not. So Proposition~\ref{prop:var_eq_sol_exists} guarantees that there exists a unique solution~$\xi(p,p_0)$. Let us show that this solution does not have conjugate points for all large enough~$p_0$. Jacobi's equation is the variational equation for the Euler--Lagrange equation~$\frac{d}{dp}f_{v'}-f_v=0$ despite the fact that we use the function $v$, which is not a solution of the Euler--Lagrange equation. Thus, the function~$\zeta(q,\alpha)=\frac{1}{p_0}\xi(p_0q)$ is a solution of the variational equation for~\eqref{eq:nu} with~$\nu(q,\alpha)=\frac{1}{p_0}v(p_0q,p_0)$. Its graph at~$\alpha=0$ is depicted in Fig.~\ref{fig:Jacobi_limit}. The function~$\zeta(q,\alpha)$ does not vanish on~$[0;p_0)$ for small enough~$\alpha$, since the function~$\zeta(q,0)$ does not vanish on the same interval (see Corollary~\ref{cor:est_Jacobi_difference}).

\medskip

It remains to verify the asymptotics stated at the beginning of the section. All of them can be easily obtained by passing to the limit~$\alpha\to+0$. We start with the first item on the list (see page~\pageref{list:asymptotics}): since~$\rho(\alpha) = \hat r + O(\alpha)$ is an analytic function of~$\alpha$ by Proposition~\ref{prop:uniform_limit}, we have
\[
r(p_0) = p_0 (\hat r + O(p_0^{-2})).
\]

\noindent The second item is obtained in a similar way: $\nu(0,\alpha) = \hat\nu(0) + O(\alpha)$. Thus, $M(p_0)=v(0,p_0) = p_0\nu(0,\alpha) = p_0 (\hat\nu(0) + O(p_0^{-2}))$ and
\[
\hat M = \hat\nu(0) = 0.315759 \pm 10^{-6}.
\]
\noindent Obviously, $\nu'(0,\alpha) = \hat\nu'(0) + O(p_0^{-2})$. Thus, $v'_0=\hat\nu'(0)$. For the value of the functional, we have
\[
J(v) = \alpha\left(\int_0^{\rho(\alpha)}g(q,\eta,\eta')\,dq + \int_{\rho(\alpha)}^1g(q,\nu,\nu')\,dq\right),
\]
\noindent where the value~$\hat J$ is found by a direct computation of the expression in brackets.

Therefore, we construct a family af extremals of Pontryagin's Maximum Principle for all large enough~$p_0$ (i.e., for~$p_0\ge P_0$) and find their asymptotics. We think that our construction can be strengthened as follows:

\begin{hyp}
	\label{hyp1}
	The trajectories described in the beginning of the present section exist, are uniquely defined, and satisfy Pontryagin's Maximum Principle~\eqref{eq:PMP} for all~$p_0>\sqrt{3}$.
\end{hyp}

This hypothesis is confirmed by the following proposition.

\begin{prop}
	The equation~$I(\rho,\alpha)=0$ has a solution $\rho(\alpha)\in(0;1)$ for all~$\alpha>\frac13$.
\end{prop}

\begin{proof}	
	Let us investigate the behavior of~$I$ as a function of~$\rho$ in neighborhoods of~$\rho=0$ and~$\rho=1$. The integral~$I$ from~\eqref{eq:I_def} has the form~$\int_0^\rho q G_\alpha(q,\rho)\,dq$, where $G_\alpha$ is an analytic function in a neighborhood of~$\rho=q=0$. Thus, $I(0,\alpha)=0$, $I'_{\rho}(0,\alpha)=0$, and~$I''_{\rho\rho}(0,\alpha)=G_\alpha(0,0)$, i.e.,
	\[
	I''_{\rho\rho}(0,\alpha) = 
	\frac{4|\nu(0,\alpha)|}{(\nu^2(0,\alpha)+\alpha)^2}
	\left[
	-\frac{\nu'^2(0,\alpha)+1}{2\nu(0,\alpha)} +
	\frac{2\nu(0,\alpha)\nu'^2(0,\alpha)}{\nu^2(0,\alpha)+\alpha}
	\right].
	\]
	\noindent This expression is negative for all small enough~$\alpha$ (since it is negative for~$\alpha=0$). Numerical computations show that this expression is negative for all~$0\le\alpha\le1$. Therefore, $I(\alpha,\rho)\sim -c\rho^2$, $c>0$ as $\rho\to +0$. 
	
	The asymptotic behavior of~$I(\rho,\alpha)$ as~$\rho\to1-0$ is more complicated. We are able to prove that~$I(\rho,\alpha)\sim c'\sqrt{1-\rho}$ as $\rho\to1-0$, and~$c'>0$ for~$p_0>\sqrt{3}$. Let us use the following relation, which holds for an arbitrary function~$\eta(p)$:
	\[
	g_\eta-g_{q\eta'}-\eta'g_{\eta\eta'} = 
	\frac{g_{\eta'\eta'}}2\frac{q\eta'-\eta}{\eta^2-q^2} - \frac{\eta'}2\frac{dg_{\eta'\eta'}}{dq}.
	\]
	\noindent Now we take the integral in~$I$ by parts. Since $\eta$ is an affine function, $\eta'=\nu'(\rho)$ and $q\eta'-\eta=\rho\nu'(\rho)-\nu(\rho)$ do not depend on~$q$. Thus,
	\begin{multline}
	\label{eq:I_main_formula}
	\frac12 I(\alpha,\rho) =
	-\nu'(\rho)q\frac{\sqrt{\eta^2-q^2}}{(\eta^2+\alpha)^2}\Big|_0^\rho +
	\nu'(\rho)\int_0^{\rho}
	\frac{\sqrt{\eta^2-q^2}}{(\eta^2+\alpha)^2}\,dq+\\
	+ (\rho\nu'(\rho)-\nu(\rho)) \int_0^{\rho}
	\frac{q\,dq}{(\eta^2+\alpha)^2\sqrt{\eta^2-q^2}}.
	\end{multline}
	
	Let us estimate the asymptotic for each of the three terms in~\eqref{eq:I_main_formula} as~$\rho\to 1-0$ and $\alpha\ne 0$. We use the convexity of~$\nu$ in a neighborhood of~$\rho=1$, which we are able to guarantee for~$0<\alpha<3$. We fix a value of~$\alpha\in(0;3)$. Let $\rho=1-\Delta\rho$, $\Delta\rho>0$. Denote $\ddot\nu_0=\nu''(0)=(3-\alpha)/(3+3\alpha)$ for short. Then
	\[
	\nu(1-\Delta\rho) = 1-\Delta\rho + \frac12\ddot\nu_0\Delta\rho^2 + O(\Delta\rho^3);
	\quad
	\nu'(1-\Delta\rho) = 1 - \ddot\nu_0\Delta\rho + O(\Delta\rho^2).
	\]
	\noindent Using these formulas, we obtain
	\[
	\eta^2(q) - q^2 = 2\ddot\nu_0\Delta\rho q(1-q) + O(\Delta\rho^2);
	\quad
	\eta^2(q) + \alpha = q^2 + \alpha + O(\Delta\rho).
	\]
	\noindent Here the terms~$O(\Delta\rho^k)$ depend on~$q$, but the fraction~$O(\Delta\rho^k)/\Delta\rho^k$ is bounded as~$\Delta\rho\to +0$ uniformly on~$q\in[0;1]$. All terms of the form~$O(\Delta\rho^k)$ will have the same property in what follows.
	
	Let us substitute all the obtained decompositions into terminant in~\eqref{eq:I_main_formula}. We have
	\[
	I_0=-\nu'(\rho)q\frac{\sqrt{\eta^2-q^2}}{(\eta^2+\alpha)^2}\Big|_0^\rho =
	-\frac{\sqrt{\ddot\nu_0}}{(1+\alpha)^2}\Delta\rho +
	O(\Delta\rho^2).
	\]
	
	For the first integral in~\eqref{eq:I_main_formula}, we can write
	\begin{multline*}
	I_1=\int_0^{1-\Delta\rho}\frac{\sqrt{\eta^2-q^2}}{(\eta^2+\alpha)^2} = 
	\sqrt{2\ddot\nu_0}\sqrt{\Delta\rho}
	\int_0^{1-\Delta\rho}\frac{\sqrt{q(1-q)+O(\Delta\rho)}}{(q^2+\alpha^2+O(\Delta\rho))^2}\,dq=\\
	=\sqrt{2\ddot\nu_0}\sqrt{\Delta\rho}
	\left[
	\int_0^1\frac{\sqrt{q(1-q)}}{(q^2+\alpha^2)^2}\,dq +
	O(\Delta\rho)
	\right].
	\end{multline*}
	
	Similarly, for the second integral in~\eqref{eq:I_main_formula}, we obtain
	\[
	I_2 = \int_0^{\rho}\frac{q\,dq}{(\eta^2+\alpha)^2\sqrt{\eta^2-q^2}} = 
	\frac{1}{\sqrt{2\ddot\nu_0}\sqrt{\Delta\rho}}
	\left[
	\int_0^1\frac{\sqrt{q}}{(q^2+\alpha)^2\sqrt{1-q}}\,dq +
	O(\Delta\rho)
	\right].
	\]
	\noindent So it remains to compute the coefficient of~$I_2$ in $I$:
	\[
	\rho\eta'(\rho)-\eta(\rho) = -\ddot\nu_0 \Delta\rho + O(\Delta\rho^2).
	\]
	\noindent Therefore,
	\[
	I(1-\Delta\rho,\alpha) = 
	\sqrt{\Delta\rho}
	\frac{\sqrt{\ddot\nu_0}}{\sqrt{2}}
	\left(
	\int_0^{1}
	\frac{\sqrt{q}(1-2q)}{(q^2+\alpha)^2\sqrt{1-q}}\,dq +
	O(\Delta\rho^{\frac32})
	\right) -
	\frac{\sqrt{\ddot\nu_0}}{(1+\alpha)^2}\Delta\rho +
	O(\Delta\rho^2).
	\]
	
	The obtained integral can easily be computed explicitly:
	\[
	\int_0^{1}
	\frac{\sqrt{q}(1-2q)}{(q^2+\alpha)^2\sqrt{1-q}}\,dq
	=\frac{\pi  \sqrt{2}}{8}
	\frac{\left(1-\alpha -\sqrt{\alpha} \sqrt{1+\alpha }\right)}
	{\alpha^{\frac54} (1+\alpha)^{\frac32}\sqrt{\sqrt{\alpha}+\sqrt{1+\alpha}}}.
	\]
	\noindent It is positive for~$\alpha<\frac{1}{3}$, i.e., for~$p_0>\sqrt{3}$. So~$I(\rho,\alpha)\sim c'\sqrt{1-\rho}$ as~$\rho\to1-0$, where~$c'>0$ for~$\alpha<\frac13$, and~$I(\rho,\alpha)\sim -c\rho^2$ as $\rho\to+0$, $c>0$ for~$\alpha<1$. Thus, $I(\rho,\alpha)$ is positive in a left neighborhood of~$\rho=1$, and it is negative in a right neighborhood of~$\rho=0$. So the equation $I(\rho,\alpha)=0$ must have a solution $\rho(\alpha)\in(0;1)$.
\end{proof}

Our hypothesis is that this solution is unique and comes from the solution~$\hat r$ of the limit equation~$I(\rho,0)=0$ at~$\alpha=0$.

\section{Proof of the optimality}
\label{sec:second_var}

We shall prove in this section that any extremal from the field constructed in Sec.~\ref{sec:field_of_extremals} is a local minimum of the functional~$J$ in the class~$PC^2$ for the corresponding height~$M=v(0)$.

\begin{thm}
	\label{thm:second_var}
	Let~$\hat v\in PC^2[0;p_0]$ be an admissible trajectory in problem~\eqref{problem:main_maxwell} on the interval~$[0;p_0]$ from the field of extremals constructed in Sec.~\ref{sec:field_of_extremals}. In this case, there exists an~$\varepsilon>0$ such that if~$h\in PC^2[0;p_0]$, $h(0)=h(p_0)=h'(p_0)=0$, $\|h''\|\le\varepsilon$, and the function~$v+h$ is convex, then~$J(\hat v+h)\ge J(\hat v)$.
\end{thm}

Let us note some important differences from the classical case: first, problem~\eqref{problem:main_maxwell} is not a problem of classical calculus of variations (since we are looking for a solutions in the class of convex curves); second, the Legendre condition vanishes, since~$\hat f_{v'v'}(p)>0$ for $p\in[0;p_0)$, but $\hat f_{v'v'}(p_0)=0$; and third, $f\to\infty$ at the right end~$p=p_0$.

The proof of the theorem  has the following structure. We start with the smoothness properties of the functional~$J$ and prove that it has the required number of derivatives. Next, we prove the nonnegativity of the first derivative, which follows from Pontryagin's Maximum Principle. Then we estimate the second derivative from below by a quadratic functional on~$h$ of special form. We finish by proving that a the third-order remainder term in the Taylor decomposition of~$J$ is bounded above by this quadratic functional.

Let us emphasize certain difficulties that appear in this scheme. The first one is related to the fact that the functional~$J$ has a singularity at the right end~$p_0$, so we are forced to make sharp estimates in working with its derivatives. The second one appears due to the fact that the Legendre condition vanishes at the right end, $\hat f_{v'v'}(p_0)=0$, so the classical estimates for the second and third derivatives of~$J$ do not work.

\medskip

Since~$f$ has a singularity at the right end, the Taylor decomposition for~$J$ does not follow from classical theory. So we start with the derivatives of~$J$ in the space~$PC^2$. The following representation is key for what follows.

\begin{prop}
	\label{prop:f_v_polinomial}
	The partial derivatives of~$f$ with respect to~$v$ have the following form:
	\begin{equation}
	\label{eq:f_v_polinomial}
	\frac{\partial^k}{\partial v^k}f = (v^2-p^2)^{-\frac{2k-1}{2}} P_k(p,v,v') + (pv'-v)(v^2-p^2)^{-\frac{2k+1}{2}}Q_k(p,v),
	\end{equation}
	\noindent where~the $P_k$ and~$Q_k$ are analytic functions that have poles only at thepoints~$v=0,\pm\mathrm{\bf i}$. Moreover, $P_k$ is a quadratic polynomial in~$v'$, and~$Q_k$ does not depend on~$v'$.
\end{prop}

\begin{proof}
	The function~$f$ has the following structure:
	\[
	f=(v^2-p^2)^\frac12 P_0(p,v,v') + (pv'-v)(v^2-p^2)^{-\frac12}Q_0(p,v)
	\]
	\noindent where~$P_0$ and~$Q_0$ possess the desired properties. A direct differentiation gives
	\[
	f_v = (v^2-p^2)^{-\frac12} P_1(p,v,v') + (pv'-v)(v^2-p^2)^{-\frac32}Q_1(p,v),
	\]
	\noindent where~$P_1=vP_0 + (v^2-p^2)(P_0)'_v-Q_0 + (pv'-v)(Q_0)'_v$ and~$Q_1=-vQ_0$, i.e., $P_1$ is a quadratic polynomial in~$v'$, and~$Q_1$ does not depend on~$v'$. It remains to repeat this process by induction.
\end{proof}

\begin{corollary}
	\label{cor:f_v_est}
	Let $v\in PC^2[0;p_0]$, $v(p_0)=p_0$, $v'(p_0)=1$, $v''(p_0)>0$, and let $v(p)>p$ for $p\in[0;p_0)$. The following estimates hold
	\[
	\begin{array}{cc}
	\left|\frac{\partial^k}{\partial v^k} f\right| \le c_k (p_0-p)^{-2k} &
	\left|\frac{\partial}{\partial v'}\frac{\partial^k}{\partial v^k} f\right| \le c_k(p_0-p)^{-2k-1}\\
	\left|\frac{\partial^2}{\partial v'^2}\frac{\partial^k}{\partial v^k} f\right| \le c_k(p_0-p)^{-2k+1} &
	\frac{\partial^3}{\partial v'^3}\frac{\partial^k}{\partial v^k} f \equiv0\\
	\end{array}
	\]
	\noindent where the constant~$c_k$ depends (monotonically\footnote{Namely, the lower written norms, the smaller are the $c_k$.}) on\footnote{Remark that all these norms are finite.}~$\|v\|$, $\|v'\|$, $\|v''\|$, $\|\frac1v\|$ and $\|\frac{(p_0-p)^2}{v-p}\|$.
\end{corollary}

\begin{proof}
	Let us use representation~\eqref{eq:f_v_polinomial}. Fix an arbitrary number~$c>0$. Consider a compact domain in~$\R^3=\{(p,v,v')\}$ given by the conditions~$p\in[0;p_0]$, $|v'|\le c$ and $\frac1c\le v\le c$. The functions~$P_k$ and~$Q_k$ are continuous on this domain; thus, they are bounded above. It remains to note that
	\begin{gather*}
	|v^2(p)-p^2|\le \frac12\|v''\|(p_0-p)^2,\\
	|v(p)-p| \ge \left\|\frac{(p_0-p)^2}{v-p}\right\|^{-1}(p_0-p)^2,\\
	|v(p)+p| \ge \left\|\frac1v\right\|^{-1},\\
	|pv(p)-v'(p)|\le p_0\|v''\|(p_0-p).
	\end{gather*}
\end{proof}

Denote by~$J^k(v)[h_1,\ldots,h_k]$ the $k$-variation of the functional~$J$, i.e.,
\[
J^0(v) = \int_0^{p_0}f\,dp,
\]
\[
J^1(v)[h_1] = \int_0^{p_0} (f_vh_1+f_{v'}h_1')\,dp,
\]
\[
J^2(v)[h_1,h_2] = \int_0^{p_0} (f_{vv}h_1h_2+f_{vv'}(h_1h_2'+h_1'h_2) + f_{v'v'}h_1'h_2')\,dp,
\]
\noindent etc.

\begin{prop}
	\label{prop:J_direct_deriv}
	Let~$v$ satisfy the conditions of Corollary~\ref{cor:f_v_est}, $h_i\in PC^2[0;p_0]$, and let $h_i(p_0)=h_i'(p_0)=0$. Then, for any~$k$, the variation~$J^k(v)[h_1,\ldots,h_k]$ is finite, continuously depends on~$v$, and
	\[
	\exists\ \frac{d}{d\lambda}\Big|_{\lambda=0}J^k(v+\lambda h_{k+1})[h_1,\ldots,h_k] = J^{k+1}(v)[h_1,\ldots,h_k,h_{k+1}].
	\]
\end{prop}

\begin{proof}
	First, we show that the $k$-variation is finite. Since~$f_{v'v'v'}\equiv0$, the integrand in~$J^k$ has only three types of terms:
	\[
	J^k=\int_0^{p_0}
	\left[
	\frac{\partial^k f}{\partial v^k} \prod_i h_i + 
	\frac{\partial^{k-1}f_{v'}}{\partial v^{k-1}} 
	\sum_i h_i'\prod_{l\ne i}h_l +
	\frac{\partial^{k-2}f_{v'v'}}{\partial v^{k-2}}
	\sum_{i<j} h_i'h_j'\prod_{l\ne i,j}h_l
	\right]\,dp.
	\]
	\noindent We know that~$|h_i|\le\frac12\|h_i''\|(p_0-p)^2$ and~$|h_i'|\le\|h_i''\|(p_0-p)$ by assumption. Thus, it follows from Corollary~\ref{cor:f_v_est} that the first term in the integrand is bounded by the constant~$2^{-k}c_k\prod_i\|h_i''\|$, the second term by the constant~$2^{k-1}kc_{k-1}\prod_i\|h_i''\|$, and the third term by the function~$2^{k-3}k(k-1)c_{k-2}\prod_i\|h_i''\| (p_0-p)^3$. Consequently, the given integral is finite.
	
	Now we prove the continuity for~$k=0$.
	\[
	J(v+h)-J(v) = \int_0^{p_0}\big[
	f(p,v,v')-f(p,v+h,v'+h')
	\big]\,dp
	\]
	\noindent Lagrange's formula for the integrand shows that it is bounded for each~$p\ne p_0$ by the function~$R=|f_v(p,w,W)||h| + |f_{v'}(p,w,W)||h'|$, where~$w(p)=v(p)+\mu(p)h(p)$, $W(p)=v'(p)+\mu(p)h'(p)$ and $\mu(p)\in[0;1]$. The function~$R$ is uniformly bounded in a neighborhood of~$v$ by the constant~$(\frac12c_1 + c_0)\|h''\|$ (this follows from corollary~\ref{cor:f_v_est}). Thus,~$J(v+h)-J(v)=O(\|h''\|)$.
	
	Let us now compute the directional derivative for~$k=0$:
	\[
	\frac{J(v+\lambda h)-J(v)}{\lambda} = 
	\int_0^{p_0}\frac{f(p,v+\lambda h,v'+\lambda h)-f(p,v,v')}{\lambda}\,dp.
	\]
	\noindent We use Lagrange's formula for the integrand again: it is bounded for each~$p$ by the function~$|f_v(p,v+\mu h,v'+\mu h)||h| + |f_{v'}(p,v+\mu h,v'+\mu h)||h'|$ for some $|\mu|\le|\lambda|$. This function is bounded for all~$p$ and small enough~$\mu$ by Corollary~\ref{cor:f_v_est}. Therefore, using Lebesgue's theorem, we can pass to the pointwise limit under the integral sign.
	
	The proof for $k\ge 1$ is similar. Namely, we again use Lagrange's formula for the integrand in the difference $J^k(v+h_k)-J^k(v)$ to prove that it is bounded above by the sum of absolute values of all terms in the integrand in~$J^{k+1}$ for some~$\mu\in[0;1]$. We again use Corollary~\ref{cor:f_v_est} to prove that all these terms are bounded for all~$p$, which leads to the continuity of~$J^k$ . The directional derivative of~$J^k$ can again be computed by Lebesgue's theorem, since its conditions are fulfilled by Corollary~\ref{cor:f_v_est}.	
\end{proof}

\begin{proof}[Proof of Theorem~\ref{thm:second_var}.]
	
	First, we prove that the functional~$J$ is differentiable in the Fr\'echet sense (in the space~$X=\{v\in PC^2[0;p_0]:v(0)=M,v(p_0)=p_0, v'(p_0)=1\}$, i.e., the variation~$h$ belongs to the space~$Y=\{h\in PC^2[0;p_0]:h(0)=h(p_0)=h'(p_0)=0\}$) and its derivative~$J'$ coincides with the first variation~$J^1$. So let~$v\in X$ be an arbitrary admissible trajectory. Let us prove that
	\[
	\forall h\in Y\qquad J(v+h) - J(v) - J^1(v)[h] = O(\|h''\|^2).
	\]
	\noindent To do this, we represent this difference in the following form:
	\begin{multline*}
	J(v+h) - J(v) - J^1(v)[h] = 
	\int_0^1 \frac{d}{d\mu}(J(v+\mu h) - \mu J^1(v)[h])\,d\mu=\\
	=\int_0^1 (J^1(v+\mu h)[h] - J^1(v)[h])\,d\mu=
	J^1(v+\mu h)[h] - J^1(v)[h].
	\end{multline*}
	\noindent The first equality holds, since the derivative of~$J$ wrt the given direction~$h$ exists and continuously depends on~$v$ by Proposition~\ref{prop:J_direct_deriv}. The last equality holds for some~$\mu\in[0;1]$, since the integrand continuously depends on~$\mu$. Next,
	\begin{multline*}
	J^1(v+\mu h)[h] - J^1(v)[h]=
	\int_0^\mu \frac{d}{d\lambda}J^1(v+\lambda h)[h]\,d\lambda=\mu J^2(v+\lambda h)[h,h]=\\
	=\mu\int_0^{p_0} \big[
	f_{v'v'}(p,\tilde v,\tilde v') h'^2 + 
	2f_{vv'}(p,\tilde v,\tilde v')hh' + 
	f_{vv}(p,\tilde v,\tilde v')h^2
	\big]\,dp,
	\end{multline*}
	\noindent where~$\tilde v=v+\lambda h$ for some $\lambda\in[0;\mu]$, which does not depend on~$p$. Using Corollary~\ref{cor:f_v_est}, we see that the integrand is bounded by the function~$(c_0(p_0-p)^3+\frac12c_1 + \frac14c_2)\|h''\|^2$ not depending on~$\lambda$. This proves the differentiability of~$J$. Similarly, we can prove that the $k$-th derivative of the functional~$J$ exists, and it coincides with the $k$-variation~$J^k$.
	
	Now let us prove that~$\hat v$ is a local minimum of the functional~$J$, i.e., there exists an~$\varepsilon>0$, such that~$J(\hat v+h)\ge J(\hat v)$ if the curve~$\hat v+h$ is convex, $h(0)=h(p_0)=h'(p_0)=0$, and~$\|h''\|\le\varepsilon$. The curve~$\hat v+h$ satisfies the inequality~$v\ge p$ automatically if the norm~$\|h''\|$ is small enough, since~$v''(p)>0$ in a neighborhood of~$p_0$.
	
	We start with the nonnegativity of the first directional derivative wrt any admissible direction. We obtain the following equality by integrating $J^1(\hat v)[h]$ by parts twice:
	\[
	J^1(\hat v)[h] = \varphi h\big|_0^{p_0} + \psi h'\big|_0^{p_0} - \int_0^{p_0}\psi h''\,dp.
	\]
	\noindent We claim that the first term vanishes, because~$h(0)=h(p_0)=h'(p_0)=0$. Indeed, the function~$\varphi$ is continuous on~$[0;p_0)$ and it has a singularity of~$c(p_0-p)^{-1}$ type at~$p_0$, since $\hat v$ is a singular extremal in a neighborhood of~$p_0$, on which~$\varphi\equiv \hat f_{v'}$. The second term vanishes by the orthogonality conditions~$\psi(0)=\psi(p_0)=0$. Let us consider the third term. We know that~$\psi\le 0$, and if~$v''(p)>0$ for some~$p$, then~$\psi(p)=0$. The function~$\hat v+h$ must be convex (since~$h$ is an admissible variation). Therefore, if~$h''(p)<0$ at a point~$p$, then~$v''(p)>0$, and thus~$\psi(p)=0$. Consequently, $\psi h''\le 0$, and we have~$J^1(\hat v)[h]\ge 0$.
	
	Consider the second derivative
	\[
	J^2(\hat v)[h,h] = \int_0^{p_0}\big[ A(p) h'^2(p) + 2B(p)h'(p)h(p) + C(p)h^2(p)\big]\,dp,
	\]
	\noindent where~$A=\hat f_{v'v'}$, $B=\hat f_{v'v}$ and~$C=\hat f_{vv}$. Note that~$A>0$ for~$p\in[0;p_0)$, but~$A(p_0)=0$. Therefore, the classical estimates from the calculus of variations do not work. Usually, if the Legendre condition does not vanish, then the second variation can be estimated from below by~$\int h'^2\,dp$. This estimate does not work for the problem under consideration. Here the lower bound is completely different. Looking ahead, we can say that the lower bound should not destroy the structure of Jacobi equation~\eqref{eq:type_of_Jacobi_eq}.
	
	So we have found a correct quadratic order of the functional$J$. Specifically, we shall prove the following lower estimate for~$J^2$ for all not necessarily admissible~$h$:
	\begin{equation}
	\label{eq:second_var_positive}
	J^2[h,h]\ge \gamma\int_0^{p_0} \big[(p_0-p)h'^2 + (p_0-p)^{-1}h^2\big]\,dp.
	\end{equation}
	\noindent Here~$\gamma>0$ is a fixed number. The quadratic functional on the right-hand side will play a key role in proving the upper estimate of the third variation by using the second one. The result will prove the optimality of~$\hat v$.
	
	The proof of estimate~\eqref{eq:second_var_positive} is based on a trick suggested by Legendre. Let us add and subtract from the left-hand side of~\eqref{eq:second_var_positive} the following term~$\int_0^{p_0}\frac{d}{dp}(\zeta h^2)dp$ with some function~$\zeta(p)$. For short, we denote ~$\tilde A=A-\gamma (p_0-p)$ and~$\tilde C = C-\gamma(p_0-p)^{-1}$. Then
	\begin{multline*}
	J^2(\hat v)[h,h]-\gamma\int_0^{p_0} \big[(p_0-p)h'^2 + (p_0-p)^{-1}h^2\big]\,dt = \\
	=\int_0^{p_0}\big[\tilde Ah'^2 + 2(B+\zeta)h'h + (\tilde C+\zeta')h^2\big]\,dp - \zeta h^2\big|_0^{p_0}=\\
	=\int_0^{p_0}\left[\tilde A\Big(h'+\frac{B+\zeta}{\tilde A}h\Big)^2 + \Big(\zeta' - \frac{(B+\zeta)^2}{\tilde A} + \tilde C\Big)h^2\right]\,dp - \zeta h^2\big|_0^{p_0}.
	\end{multline*}
	\noindent We claim that the first term in the written integral is nonnegative. Indeed, $A(p)\ge 0$ and the equality holds only at the right end. Moreover,
	\[
	A(p_0-\Delta p) = \left(\frac{4\sqrt{p_0\hat v''(0)}}{1+p_0^2}+o(1)\right)\Delta p.
	\]
	\noindent Thus, if~$\gamma>0$ is small enough, then~$\tilde A\ge 0$. Therefore, if we find a function~$\zeta$ such that the second term in the integrand and the terminant both vanish, then this will prove the positive definiteness of the second variation~$J^2$ in terms of the quadratic functional in~\eqref{eq:second_var_positive}.
	
	The Jacobi equation~$\xi''=-\frac{A'}{A}\xi' + \frac{C-B'}{A}\xi$ has the form~\eqref{eq:type_of_Jacobi_eq} with~$\lambda=-\frac14$ and $\sigma\equiv0$. Thus, it has a unique solution with initial conditions~$\xi(p_0)=0$ and~$\xi'(p_0)=1$ by Proposition~\ref{prop:var_eq_sol_exists}, and this solution does not vanish on~$[0;p_0)$ as shown in Sec.~\ref{sec:field_of_extremals}. Consider the modified Jacobi equation~$\tilde\xi''=-\frac{\tilde A'}{\tilde A}\tilde\xi' + \frac{\tilde C-B'}{\tilde A}\tilde\xi$, which also has the form~\eqref{eq:type_of_Jacobi_eq} with the same~$\lambda=-\frac14$ and~$\sigma\equiv0$. Therefore, the modified equation also has a unique solution with the same initial data by Proposition~\ref{prop:var_eq_sol_exists}. Using Corollary~\ref{cor:est_Jacobi_difference}, we can estimate the difference between the second derivatives of~$\xi$ and~$\tilde\xi$ by a value of order~$O(\gamma)$. Thus,~$\tilde\xi$ also does not vanish on~$[0;p_0)$ if~$\gamma$ is small enough. 
	
	Now we prove estimate~\eqref{eq:second_var_positive} by putting~$\zeta = -\frac{\tilde A\tilde \xi'}{\tilde\xi} - B$. First (I), let us check the continuity of~$\zeta$. Since~$\tilde\xi\in C^2$ and $\tilde\xi(p)\ne 0$ for $p\in[0;p_0)$, we have $\zeta\in C^1[0;p_0)$. Second (II), let us check that the function~$\zeta$ satisfies the following Riccati equation:
	\[
	\zeta' = \frac{(B+\zeta)^2}{\tilde A} - \tilde C,
	\]
	\noindent Indeed,
	\[
	\zeta'\tilde\xi + \zeta\tilde\xi' = \frac{d}{dp}(\zeta\tilde\xi) = -B\tilde\xi'-\tilde C\tilde\xi
	\quad\Rightarrow\quad
	\zeta'=-(\zeta+B)\frac{\tilde\xi'}{\tilde\xi}-\tilde C = \frac{(\zeta+B)^2}{\tilde A} - \tilde C.
	\]
	\noindent Third (III), let us determine the asymptotic of~$\zeta h^2$ at~$p_0$ from the left. A direct computation gives
	\[
	B(p_0-\Delta p) = \left(\frac{p_0}{(1+p_0^2)(p_0\hat v''(0))^{3/2}}+o(1)\right)\Delta p^{-3}
	\]
	\noindent Since~$\tilde\xi(p_0)=0$ and~$\tilde\xi'(p_0)\ne 0$, we have $\tilde\xi'/\tilde\xi = -(1+o(1))\Delta p^{-1}$. So
	\[
	\zeta(p_0-\Delta p) = -\left(\frac{p_0}{(1+p_0^2)(p_0\hat v''(0))^{3/2}}+o(1)\right)\Delta p^{-3},
	\]
	\noindent and $\zeta h^2\to 0$ as $\Delta p\to 0$, since $h(p_0)=h'(p_0)=0$.
	
	Putting together the first (I), second (II), and third (III) facts about~$\zeta$ we prove estimate~\eqref{eq:second_var_positive}.
	
	Now we prove that there exists an~$\varepsilon>0$ such that if $\|h''\|\le\varepsilon$, then $J(\hat v+h)\ge J(\hat v)$. Obviously,
	\begin{multline*}
	J(\hat v+h) - J(\hat v) - J^1(\hat v)[h] - \frac12 J^2(\hat v)[h,h] = \frac16J^3(v)[h,h,h] = \\
	= \frac16\int_0^{p_0} \big[
	f_{v'v'v'}h'^3 + 3f_{v'v'v}h'^2h + 3f_{v'vv}h'h^2 + f_{vvv}h^3
	\big]\,dp,
	\end{multline*}
	\noindent where we must substitute~$v=\hat v+\lambda h$ for some~$\lambda\in[0;1]$ into the right-hand side. We shall prove that if~$\|h''\|$ is small enough, then the right-hand side is bounded above by the second variation, i.e., there exists a constant~$c>0$, such that
	\begin{equation}
	\label{eq:J_third_estimate}
	\boxed{J^3(v)[h,h,h]\le c\|h''\|J^2(\hat v)[h,h]}
	\end{equation}
	\noindent for all $h$ with small enough norm~$\|h''\|$.
	
	We shall use estimate~\eqref{eq:second_var_positive}. The first term on the right-hand side disappears, since~$f_{v'v'v'}\equiv0$. For the second term we have~$|f_{v'v'v}|\le c_1(p_0-p)^{-1}$ by Corollary~\ref{cor:f_v_est}. Since~$|h|\le \frac12 (p_0-p)^2\|h''\|$, we have $|f_{v'v'v}h'^2h|\le \frac12 c_1\|h''\|(p_0-p)h'^2$, and the last term is bounded by estimate~\eqref{eq:second_var_positive}. Let us proceed to the last two terms. Unfortunately, both~$f_{v'vv}$ and~$f_{vvv}$ have harsh singularities at~$p_0$ of orders~$(p_0-p)^{-5}$ and~$(p_0-p)^{-6}$, respectively. But the frightening view of the singularities is the result of a cursory examination. Let us make it more precise. We eliminate these singularities by integrating the term~$3f_{v'vv}h'h^2$ by parts:
	\[
	\int_0^{p_0} \big[
	3f_{v'vv}h'h^2 + f_{vvv}h^3
	\big]\,dp = 
	\int_0^{p_0} \big[
	(f_{vvv} - \frac{d}{dp}f_{v'vv})h^3
	\big]\,dp + 
	f_{v'vv}h^3|_0^{p_0}.
	\]
	\noindent The terminant vanishes, since~$f_{v'vv}$ is continuous at the left end~$p=0$ and~$h(0)=0$, and~$|f_{v'vv}|\le c_0(p_0-p)^{-5}$ at the right end and $|h(p)|\le \frac12\|h''\|(p_0-p)^2$, i.e.,~$|f_{v'vv}h^3|\le c(p_0-p)$. So let us estimate the integrand. A direct computation gives
	\[
	f_{vvv} - \frac{d}{dp}f_{v'vv} = 
	\frac{R(p,v,v',v'')}{(1+v^2)^5(v^2-p^2)^{\frac32}} + 
	6\frac{2 p^3 v'-v^3 (1+v'^2)}{(1+v^2)^2(v^2-p^2)^{\frac52}},
	\]
	\noindent where~$R$ is a polynomial in~$p$, $v$, $v'$ and~$v''$, the explicit form is of no importance for us. Let us estimate the numerator of the second fraction. Since it vanishes at~$p_0$, it is an integral of its derivative:
	\[
	2 p^3 v'-v^3 (1+v'^2) = -\int_p^{p_0} \big[
	3(2p^2-v^2-v^2v'^2) + 2v'' (p^3-v^3v')
	\big]\,dp.
	\]
	\noindent Here both terms in the integrand also vanish at~$p_0$. Thus,
	\[
	|2p^2-v^2-v^2v'^2|=\Big|\int_{p_0}^p (4p - 2vv' -2vv'^3-2v^2v'v'')\,dp\Big|	\le 2(4p_0+\|v\|\|v'\|(1+\|v'\|^2+\|v\|\|v''\|))(p_0-p);
	\]
	\[
	|p^3-v^3v'| = \Big|\int_{p_0}^p (3p^2-3v^2v'^2-v^3v'')\,dp\Big|\le(3p_0^2 + \|v\|^2(3\|v'\|^2+\|v\|\|v''\|))(p_0-p).
	\]
	
	Therefore, the estimated numerator is bounded by the function~$c(p_0-p)^2$, where the constant~$c$ can be chosen independently of the norm~$\|h''\|$ (if the latter is small enough). Consequently, $|f_{vvv} - \frac{d}{dp}f_{v'vv}|\le c(p_0-p)^{-3}$ if we increase~$c$. Thus,
	\[
	\left|\Big(f_{vvv} - \frac{d}{dp}f_{v'vv}\Big)h^3\right| \le \frac12 c\|h''\| (p_0-p)^{-1}h^2.
	\]
	\noindent Here we agin use estimate~\eqref{eq:second_var_positive}, which proves the key estimate~\eqref{eq:J_third_estimate} for the remainder term of third variation. 
	
	So there exists an~$\varepsilon>0$ and~$C>0$ such that if~$\|h''\|\le\varepsilon$, $h(0)=h(p_0)=h'(p_0)=0$, and the function~$v+h$ is convex, then 
	\[
	\boxed{J(\hat v+h)\ge J(\hat v) + C\int_0^{p_0}\big[(p_0-p) h'^2 + (p_0-p)^{-1}h^2\big]\,dp}
	\]
	\noindent and so $J(\hat v+h)\ge J(\hat v)$. Quod erat demonstrandum!
\end{proof}

\section{Conclusions}

Let us give a full description of found solutions in Newton's aerodynamic problem. The shape of the body in problem~\eqref{problem:start} is constructed by the following steps. We start with the solution~$v(p)$ on~$[0;p_0]$ of the Euler--Lagrange equation~\eqref{eq:main_Lagrange_eq} for singular extremals with initial conditions~$v(p_0)=p_0$, $v'(p_0)=1$. Then it is necessary to change~$v$ on~$[0;r(p_0)]$ to a linear function with the condition of continuous junction of~$v$ and~$v'$ at~$r(p_0)$. The switching point~$r(p_0)$ is found from condition~\eqref{eq:I_def} (where the corresponding linear function~$v$ is substituted into the integral~\eqref{eq:I_def}). The result is an extremal in the key optimal control problem~\eqref{eq:maxwell_integral_value} (for a more detailed construction of the extremals for all large enough~$p_0$, see Sec.~\ref{sec:field_of_extremals}). Next,
\begin{enumerate}[label=(\roman*)]
	\item the constructed extremal~$v$ is prolonged by the line~$v=p$ for~$p\ge p_0$;
	\item it is reflected in the plane~$\R^2=\{(v,p)\}$ wrt the vertical line~$p=0$. The result is a convex function~$v(p)$ defined for all~$p\in\R$, which is $C^1$-smooth on~$\R\setminus\{0\}$ and has a corner at~$0$;
	\item the conjugate function~$v^*(x_1)=\sup_{p}(px_1-v(p))$ is the intersection of the boundary of the desired 3D convex body with the vertical symmetry plane~$\{x_2=0\}$. The desired body $\hat u_M\in E_M$ is the convex hull of this curve and the unit circle~$\Omega$ lying in the base.
\end{enumerate}

The curve~$v^*(x_1)$ has height~$M=v(0)>0$. It has a horisontal segment on the lower bounding plane~$\{z=-M\}$. This segment has length~$2v'(+0)$. The derivative of~$v^*$ at the ends of this segment has jumps that equal~$r(p_0)$. The one-side derivatives at the points~$\pm1$ are equal in absolute value to~$|(v^*)'(\pm1\mp0)|=p_0$. The resistance~$\mathcal{J}$ of the constructed body $\hat u_M$ equals~$2J(v)$. We know two papers~\cite{LachardNumeric,Wachsmuth2014} where some numerical computations were done for Newton's aerodynamic problem for some heights~$M\le1.5$. Note that the numerical result in the Table~\ref{tab1} for~$M=1.5$ gives the value of the functional~$\mathcal{J}$, which agrees well with numerical computations in the above-mentioned papers: it is~$\sim2\times10^{-4}$ less than the result in~\cite{LachardNumeric}, and it is~$\sim10^{-3}$ greater than the result in~\cite{Wachsmuth2014}. Numerical computations in both papers were done by discretization by the infimum of hyperplanes from a large family\footnote{Another method was also used in~\cite{Wachsmuth2014} for heights~$M\le 1.4$.}. This choice of heights in~\cite{LachardNumeric,Wachsmuth2014} is connected with the following fact: for small values of the height~$M$, the optimal solution in the class~$E_M$ is definitely not globally optimal in the original class~$C_M$, since gradients outside the symmetry plane become less than~$1$ in absolute value (see~\cite[Theorem~2.3]{Buttazzo1995}). For large enough $M$ due to numerical experiment in~\cite{LachardNumeric,Wachsmuth2014} the lower set $u^{-1}(-M)$ of the optimal solution $u$ in $C_M$ is a segment, and the body contains corners along some 1-dimensional curve in the vertical plane of symmetry $\{x_2=0\}$. But it may happen that the optimal body side boundary contains additional corners.

\bigskip

So we have found localy optimal shape of bodies in the class of convex bodies with a vertical palne of symmetry having a smooth side boundary (see Theorem~\ref{thm:second_var}).

\setlength{\tabcolsep}{25pt}
\renewcommand{\arraystretch}{1.2}
\begin{table}
	\centering
	\begin{tabular}{lllll}
		\Xhline{2\arrayrulewidth}\\[-5mm]
		\ $M$  &\ \ \ $p_0$   &\ \ $r(p_0)$ &\ \ $v'(+0)$ &\ \ \ \ \ $J$\\[2mm]
		\Xhline{4\arrayrulewidth}\\[-5mm]
		$0.5$  & $2.43337$ & $1.33559$ & $0.744669$ & $1.06309\times10^0$ \\
		$1.0$  & $3.71647$ & $1.22077$ & $0.632450$ & $5.97791\times10^{-1}$ \\
		$1.5$  & $5.14856$ & $1.19669$ & $0.586444$ & $3.50482\times10^{-1}$ \\
		$2.0$  & $6.64354$ & $1.23585$ & $0.564900$ & $2.22512\times10^{-1}$ \\
		$2.5$  & $8.16986$ & $1.31540$ & $0.553467$ & $1.51524\times10^{-1}$ \\
		$5.0$  & $15.9653$ & $1.96456$ & $0.536348$ & $4.14500\times10^{-2}$ \\
		$10.0$ & $31.7371$ & $3.57283$ & $0.531668$ & $1.06143\times10^{-2}$ \\
		$50.0$ & $158.373$ & $17.2830$ & $0.530132$ & $4.27905\times10^{-4}$ \\
		$100.$ & $316.727$ & $34.5295$ & $0.530084$ & $1.07002\times10^{-4}$ \\[2mm]
		\Xhline{2\arrayrulewidth}
	\end{tabular}
	\caption{Numerical computations for the parameters of curves~$v$ with different heights~$M$.}
	\label{tab1}
\end{table}

\section*{Acknowledgments}

The authors would like to express deep gratitude to Professors V.V.~Palin, A.S.~Kochurov, and G.~Wachsmuth.

\printbibliography

\end{document}